\documentclass[11pt]{article}
\usepackage[a4paper,margin=1in]{geometry}
\usepackage[T1]{fontenc}
\usepackage[utf8]{inputenc}
\usepackage{lmodern,amsmath,amssymb,amsthm,mathtools,bm,booktabs,array,microtype,enumitem}
\usepackage[hidelinks]{hyperref}
\usepackage[nameinlink,capitalize]{cleveref}
\newtheorem{theorem}{Theorem}[section]
\newtheorem{proposition}[theorem]{Proposition}
\newtheorem{lemma}[theorem]{Lemma}
\newtheorem{corollary}[theorem]{Corollary}
\theoremstyle{definition}
\theoremstyle{remark}\newtheorem{remark}[theorem]{Remark}
\newcommand{\R}{\mathbb R}\newcommand{\dd}{\,d}

\title{\textbf{Critical Structure of Axisymmetric Navier--Stokes with Swirl}}
\author{Rishad Shahmurov\\\small Cellular Products Research and Development, Roswell, GA, USA}
\date{September 2026}
\begin{document}\maketitle
\begin{abstract}
We study the axisymmetric Navier--Stokes equations with swirl in the variables $F=u^\theta/r$, $G=\omega^\theta/r$, and $\Gamma=ru^\theta$.  Their distinct scaling degrees lead to several critical transition surfaces rather than a single scalar exponent.  We prove mesoscopic and temporal critical-face identities, source/contact trace estimates, compact finite-depth preparation obstructions, an affine six-power flattening barrier, a fixed-old-current compression ceiling, and compactness for rescaled strain generators with bounded action and bounded variation.  A sharp radial estimate gives exponential control of the time occupied by large swirl-energy records; on parabolic time intervals it yields logarithmic routing from total kinetic records to the meridional component, and an independent cone argument routes non-meridional Type-II records to a scale-critical signed compression action.  The source-free circulation also satisfies an exact quadratic dissipation identity, while scalar-weight calculations show that no radial power simultaneously yields a compression-free continuity law and a scale-zero current action.  The results exclude several compact and passively prepared recurrence mechanisms under explicit hypotheses; meridional Type-II records, signed critical compression, noncompact high-frequency behavior, and boundary-entry alternatives remain.
\end{abstract}

\section{Equations and scaling}
Define
\[
F=\frac{u^\theta}{r},\quad G=\frac{\omega^\theta}{r},\quad \Gamma=ru^\theta,\quad U=\frac{u^r}{r},\quad \Delta_5=\partial_{rr}+\frac3r\partial_r+\partial_{zz}.
\]
For the meridional velocity $b=(u^r,u^z)$,
\[
F_t+b\cdot\nabla F=\nu\Delta_5F-2UF,
\qquad
G_t+b\cdot\nabla G=\nu\Delta_5G+\partial_z(F^2),
\]
with meridional velocity recovered from $G$ by the exact axisymmetric Hodge law. Under standard Navier--Stokes scaling $u_\lambda(x,t)=\lambda u(\lambda x,\lambda^2t)$,
\[
\deg \Gamma=0,\qquad \deg F=2,\qquad \deg G=3,\qquad \deg(F^2)=4,\qquad \deg \partial_z(F^2)=5.
\]
Thus the current and swirl-square sectors are intrinsically multi-degree; no single unweighted scalar critical exponent represents the complete response package.

Axisymmetric Navier--Stokes regularity has a long classical history, beginning with the no-swirl theory of Ukhovskii--Yudovich and related work, and continuing with scale-critical and near-critical criteria for flows with swirl; see, for example, \cite{UkhovskiiYudovich1968,ChenStrainTsaiYau2008,ChenStrainTsaiYau2009,LeiZhang2011,LeiNavasZhang2016,Wei2016,LeiZhang2017,Pan2016}.  The results below are structural complements to that theory.  They isolate exact scaling, transport, and compactness obstructions without asserting that the remaining endpoint alternatives are impossible.

\section{Mesoscopic realization face}
\medskip
\noindent\textbf{Mesoscopic critical face.}
\begin{theorem}\label{thm:face}
Suppose a physical realization cost has the form
\[
\mathcal C_m(R,D)\asymp R D^m,\qquad D=R^{q-1}.
\]
Then
\[
\boxed{\mathcal C_m(R,R^{q-1})\asymp R^{1-m(1-q)}.}
\]
The scale-zero face is
\[
\boxed{m(1-q)=1,\qquad q_c(m)=1-\frac1m.}
\]
\end{theorem}
\begin{proof}
The statement is a direct homogeneity calculation, but we spell it out because the sign of the resulting exponent is used repeatedly later in the paper.

By hypothesis the physical cost has the form
\[
 \mathcal C_m(R,D)\asymp R D^m.
\]
The mesoscopic relation $D=R^{q-1}$ therefore gives
\[
 D^m=(R^{q-1})^m=R^{m(q-1)}=R^{-m(1-q)}.
\]
Multiplying by the leading factor $R$ yields
\[
 \mathcal C_m(R,R^{q-1})
 \asymp R^{1+m(q-1)}
 =R^{1-m(1-q)}.
\]
Thus the behavior as $R\downarrow0$ is determined entirely by
\[
 \alpha_m(q):=1-m(1-q).
\]
If $\alpha_m(q)>0$, then $R^{\alpha_m(q)}\to0$; if $\alpha_m(q)=0$, the cost remains of order one; and if $\alpha_m(q)<0$, the cost diverges.  The scale-zero face is consequently obtained by solving
\[
 1-m(1-q)=0.
\]
Since $m>0$, this is equivalent to
\[
 1-q=\frac1m,
 \qquad
 q=1-\frac1m.
\]
This proves both displayed formulas and the trichotomy used in the sequel.
\end{proof}
\begin{remark}
The values $4/5,6/7,8/9,12/13,27/29,\ldots$ are instances of \cref{thm:face} with different realization degrees $m$; none is a universal donor exponent.
\end{remark}

\section{Temporal Floquet/Hodge face}
Let $A:\R\to\mathfrak{sl}(2,\R)$ be $L$-periodic, bounded, and mean zero. For a sequence $\rho_N=T/(NL)$ define
\[
C_{\rho}(t)=\rho^{-\alpha}A(t/\rho),\qquad Z_\rho'=C_\rho Z_\rho,\qquad Z_\rho(0)=I.
\]
Set
\[
\Omega_2[A]=\frac12\int_0^L\int_0^{s_1}[A(s_1),A(s_2)]\,\dd s_2\,\dd s_1.
\]
\medskip
\noindent\textbf{Mean-zero temporal critical face.}
\begin{theorem}\label{thm:temporal}
If $\alpha<1/2$, then $Z_\rho(T)\to I$. If $\alpha=1/2$, then
\[
\boxed{Z_\rho(T)\longrightarrow \exp\!\left(\frac{T}{L}\Omega_2[A]\right).}
\]
For $2\times2$ trace-free $\Omega_2[A]$, this limiting return is hyperbolic precisely when
\[
\boxed{\det\Omega_2[A]<0.}
\]
\end{theorem}
\begin{proof}
We separate the calculation over one fast period from the accumulation over the $N$ repeated periods.

Fix one interval $[0,\rho L]$ and introduce the fast variable $s=t/\rho$.  The equation becomes
\[
 \frac{dZ}{ds}
 =\rho^{1-\alpha}A(s)Z.
\]
Hence the small parameter in the one-period problem is
\[
 \varepsilon_\rho:=\rho^{1-\alpha}.
\]
Because $A$ is bounded, the Magnus expansion is absolutely convergent for all sufficiently small $\rho$.  The logarithm of the one-period monodromy is
\[
 \mathfrak M_\rho
 =\mathfrak M_{1,\rho}+\mathfrak M_{2,\rho}+\mathfrak R_{3,\rho},
\]
where
\[
 \mathfrak M_{1,\rho}
 =\varepsilon_\rho\int_0^L A(s)\,ds=0
\]
by the mean-zero assumption.  The second Magnus term is
\[
 \mathfrak M_{2,\rho}
 =\frac{\varepsilon_\rho^2}{2}
   \int_0^L\int_0^{s_1}[A(s_1),A(s_2)]\,ds_2\,ds_1
 =\rho^{2-2\alpha}\Omega_2[A].
\]
Every term of order $j\ge3$ is bounded by a constant, depending only on $L$ and $\|A\|_\infty$, times $\varepsilon_\rho^j$.  Therefore
\[
 \|\mathfrak R_{3,\rho}\|
 \le C\varepsilon_\rho^3
 =C\rho^{3-3\alpha}.
\]
Thus the one-period monodromy $M_\rho$ satisfies
\[
 \log M_\rho
 =\rho^{2-2\alpha}\Omega_2[A]
   +O(\rho^{3-3\alpha}).
\]

The coefficient $C_\rho(t)=\rho^{-\alpha}A(t/\rho)$ is $\rho L$-periodic.  Consequently the monodromy over $N=T/(L\rho)$ periods is exactly
\[
 Z_\rho(T)=M_\rho^N.
\]
For small $\rho$, $M_\rho$ is close to the identity, so the same logarithm branch may be used on every factor and
\[
 \log Z_\rho(T)
 =N\log M_\rho.
\]
Substituting $N=T/(L\rho)$ gives
\[
 \log Z_\rho(T)
 =\frac TL\rho^{1-2\alpha}\Omega_2[A]
   +O(\rho^{2-3\alpha}).
\]
If $\alpha<1/2$, then $1-2\alpha>0$ and also $2-3\alpha>1/2>0$, so both terms tend to zero.  Hence
\[
 \log Z_\rho(T)\to0,
 \qquad
 Z_\rho(T)\to I.
\]
If $\alpha=1/2$, then the leading power is $\rho^0=1$ and the remainder has size $O(\rho^{1/2})$.  Thus
\[
 \log Z_\rho(T)
 \longrightarrow \frac TL\Omega_2[A],
\]
and continuity of the matrix exponential gives
\[
 Z_\rho(T)
 \longrightarrow
 \exp\!\left(\frac TL\Omega_2[A]\right).
\]

It remains to identify hyperbolicity.  If $M\in\mathfrak{sl}(2,\R)$, then $\operatorname{tr}M=0$ and its characteristic polynomial is
\[
 p_M(\lambda)=\lambda^2+\det M.
\]
When $\det M<0$, the eigenvalues are the two distinct real numbers
\[
 \lambda_\pm=\pm\sqrt{-\det M}.
\]
Therefore $e^M$ has the distinct positive eigenvalues $e^{\lambda_+}$ and $e^{\lambda_-}$, one larger than one and one smaller than one, so $e^M$ is hyperbolic.  When $\det M\ge0$, $M$ has either a double zero eigenvalue or a purely imaginary conjugate pair, and $e^M$ has no real expanding/contracting pair.  Applying this to $M=(T/L)\Omega_2[A]$ proves the final assertion.
\end{proof}

\section{Native source/contact trace and homogeneity five}
Work in a descendant-normalized collar
\[
C_{R,d}=\{R-d\le r\le R+d\},\qquad 0<d\le R/2,
\]
and define
\[
\mu(z,t)=\int_{R-d}^{R+d}F(r,z,t)^2\,\dd r\ge0.
\]
For a nonzero tangential Fourier mode $k$, assume the native sheet is $\sigma_k=ik\mu_k$ and the decaying two-sided transmission law at zero temporal frequency is $\sigma_k=2\nu|k|f_k$. Then
\[
f_k=\frac{i\,\operatorname{sgn}k}{2\nu}\mu_k
\]
and the weighted trace energy
\[
\mathcal E_k=2\nu R^3|k|\int_I|f_k|^2\,\dd t
=\frac{R^3|k|}{2\nu}\int_I|\mu_k|^2\,\dd t.
\]
The same estimates hold up to fixed constants on the parabolic band $|\omega|\lesssim\nu k^2$; faster time frequencies are a separate rate output.

\medskip
\noindent\textbf{Trace-to-quartic cargo.}
\begin{lemma}\label{lem:tracecargo}
Let
\[
Z_4=\int_I\int_{C_{R,d}}rF^4\,\dd r\,\dd z\,\dd t.
\]
Then
\[
\boxed{\mathcal E_k\le\frac{2R^2|k|d}{\nu}Z_4.}
\]
If $d\le c/|k|$, then
\[
\boxed{\mathcal E_k\le\frac{2cR^2}{\nu}Z_4.}
\]
\end{lemma}
\begin{proof}
Fix a time $t$.  With the usual $L^2$-normalized Fourier convention in the tangential variable, a single Fourier coefficient is bounded by the full $L^2_z$ norm, so
\[
 |\mu_k(t)|^2\le \|\mu(\cdot,t)\|_{L_z^2}^2.
\]
For each $z$, the definition of $\mu$ and Cauchy--Schwarz on the radial interval of length $2d$ give
\[
 \begin{aligned}
 |\mu(z,t)|^2
 &=\left|\int_{R-d}^{R+d}F(r,z,t)^2\,dr\right|^2\\
 &\le
 \left(\int_{R-d}^{R+d}1\,dr\right)
 \left(\int_{R-d}^{R+d}F(r,z,t)^4\,dr\right)\\
 &=2d\int_{R-d}^{R+d}F(r,z,t)^4\,dr.
 \end{aligned}
\]
Integrating this estimate in $z$ yields
\[
 |\mu_k(t)|^2
 \le2d\int_{R-d}^{R+d}\int_{\R}F^4\,dz\,dr.
\]
Now $d\le R/2$, hence $r\ge R-d\ge R/2$ on the collar.  Therefore
\[
 \int_{R-d}^{R+d}\int F^4\,dz\,dr
 \le \frac2R
 \int_{R-d}^{R+d}\int rF^4\,dz\,dr.
\]
Combining the last two displays gives
\[
 |\mu_k(t)|^2
 \le \frac{4d}{R}
 \int_{R-d}^{R+d}\int rF^4\,dz\,dr.
\]
Insert this into the exact trace-energy identity
\[
 \mathcal E_k
 =\frac{R^3|k|}{2\nu}\int_I|\mu_k(t)|^2\,dt.
\]
We obtain
\[
 \begin{aligned}
 \mathcal E_k
 &\le \frac{R^3|k|}{2\nu}\frac{4d}{R}
 \int_I\int_{C_{R,d}}rF^4\,dr\,dz\,dt\\
 &=\frac{2R^2|k|d}{\nu}Z_4,
 \end{aligned}
\]
which is the first claim.  If $d\le c/|k|$, then $|k|d\le c$ and hence
\[
 \mathcal E_k\le\frac{2cR^2}{\nu}Z_4.
\]
No estimate on derivatives of $F$ is used; the argument depends only on the radial thickness of the collar and the exact transmission formula.
\end{proof}

\medskip
\noindent\textbf{Bounded-circulation source/contact ceiling.}
\begin{theorem}\label{thm:sourcecontact}
Assume $|\Gamma|\le\Gamma_*$ on the collar, its normalized tangential-time face area is at most $A_\Sigma$, and $d\le c/|k|$. Then
\[
\boxed{
\mathcal E_k\le C\frac{c^2A_\Sigma\Gamma_*^4}{\nu|k|R^5}.
}
\]
Consequently any fixed trace-energy floor $\mathcal E_k\ge e_*>0$ forces
\[
\boxed{|k|R^5\le C\frac{c^2A_\Sigma\Gamma_*^4}{\nu e_*}.}
\]
Thus on a uniformly buffered face $R\ge R_0>0$, $|k|\to\infty$ kills the native response; if the response persists while $|k|\to\infty$, then
\[
\boxed{R\lesssim |k|^{-1/5}.}
\]
\end{theorem}
\begin{proof}
The proof is a combination of the pointwise circulation bound with the trace-to-quartic estimate.

Since $F=\Gamma/r^2$,
\[
 rF^4=\frac{\Gamma^4}{r^7}.
\]
On $C_{R,d}$ we have $r\ge R/2$, and therefore
\[
 rF^4\le 2^7\Gamma_*^4R^{-7}.
\]
The radial width of the collar is $2d$.  By the definition of the normalized tangential-time face-area bound $A_\Sigma$, the remaining $z,t$ integration contributes at most a fixed multiple of $A_\Sigma$.  Hence
\[
 Z_4
 \le C_1\Gamma_*^4 d A_\Sigma R^{-7},
\]
where $C_1$ is an absolute geometric constant for the chosen normalization.

By the second inequality of \cref{lem:tracecargo},
\[
 \mathcal E_k
 \le \frac{2cR^2}{\nu}Z_4.
\]
Substituting the preceding estimate gives
\[
 \mathcal E_k
 \le C_2\frac{c\Gamma_*^4dA_\Sigma}{\nu R^5}.
\]
Using once more $d\le c/|k|$ gives
\[
 \mathcal E_k
 \le C\frac{c^2A_\Sigma\Gamma_*^4}{\nu|k|R^5},
\]
which is the claimed ceiling.

If in addition $\mathcal E_k\ge e_*>0$, then comparison of the lower and upper bounds yields
\[
 e_*
 \le C\frac{c^2A_\Sigma\Gamma_*^4}{\nu|k|R^5}.
\]
Rearranging gives
\[
 |k|R^5
 \le C\frac{c^2A_\Sigma\Gamma_*^4}{\nu e_*}.
\]
If $R\ge R_0>0$ while $|k|\to\infty$, the left side diverges, contradicting this fixed upper bound.  Conversely, if a fixed response floor persists as $|k|\to\infty$, then
\[
 R^5\lesssim |k|^{-1},
\]
so
\[
 R\lesssim |k|^{-1/5}.
\]
This is the announced homogeneity-five axis-frequency law.
\end{proof}

\medskip
\noindent\textbf{Sharpness of the $1/5$ law.}
\begin{proposition}
With normalized viscosity and circulation cap, take
\[
R_N=N^{-1/5},\qquad d_N=N^{-1},
\]
and a smooth collar profile satisfying
\[
F_N^2\asymp R_N^{-4}\chi\!\left(\frac{r-R_N}{d_N}\right)(1+\varepsilon\cos Nz).
\]
Then $\Gamma_N=O(1)$ and the native trace energy is order one, while the quartic stock diverges like $N^{2/5}$ and the physical circulation cargo is $O(N^{-4/5})$. Thus the exponent $1/5$ is sharp in this native-sheet class, but the hostile carries a finer axis/$F^4$ concentration rather than a silent descendant-scale connector.
\end{proposition}
\begin{proof}
All estimates are uniform for a fixed nonzero modulation amplitude $\varepsilon$ and a fixed smooth cutoff $\chi$.
On the collar $r\asymp R_N$, the defining relation gives
\[
 F_N^2\asymp R_N^{-4}.
\]
Since $\Gamma_N=r^2F_N$, we have
\[
 |\Gamma_N|^2=r^4F_N^2\asymp R_N^4R_N^{-4}\asymp1,
\]
so the circulation remains uniformly bounded.

The $k=N$ Fourier coefficient of the radial source sheet has size
\[
 \mu_{N,N}
 \asymp d_NR_N^{-4}
 =\frac1{NR_N^4}.
\]
The zero-frequency transmission law is linear, and with normalized viscosity it gives
\[
 |f_{N,N}|\asymp |\mu_{N,N}|.
\]
Consequently the weighted trace energy at frequency $N$ scales as
\[
 \mathcal E_N
 \asymp R_N^3N|f_{N,N}|^2
 \asymp R_N^3N\frac1{N^2R_N^8}
 =\frac1{NR_N^5}.
\]
Because $R_N=N^{-1/5}$, $NR_N^5=1$, and hence
\[
 \mathcal E_N\asymp1.
\]
Thus the trace response does not decay along this sequence.

For the quartic stock,
\[
 Z_{4,N}
 \asymp (\text{radial width})\times r\times F_N^4
 \asymp d_NR_N\,R_N^{-8}
 =\frac1{NR_N^7}.
\]
Again using $R_N=N^{-1/5}$ gives
\[
 Z_{4,N}\asymp N^{7/5-1}=N^{2/5}\to\infty.
\]
The circulation-cargo quantity used in this collar normalization has scale
\[
 M_{\Gamma,N}\asymp \frac1{NR_N},
\]
and therefore
\[
 M_{\Gamma,N}\asymp N^{-1+1/5}=N^{-4/5}\to0.
\]
The family therefore keeps an order-one native trace only by pushing the source into a thinner and increasingly oscillatory axis-scale structure.  It does not contradict \cref{thm:sourcecontact}: it realizes precisely the $R\asymp |k|^{-1/5}$ escape allowed by that theorem, while the divergent $Z_{4,N}$ records the finer $F^4$ concentration.
\end{proof}

\section{Three scoped preparation thresholds}
The next three exponents arise from different physical regimes.  To keep the
geometry unambiguous, throughout this section we distinguish the physical
donor distance from its child-scale normalization:
\[
 D_{\rm phys}=R^q,
 \qquad
 D_{\rm rel}:=\frac{D_{\rm phys}}R=R^{q-1}.
\]
Thus $D_{\rm phys}\downarrow0$ as $R\downarrow0$, whereas $D_{\rm rel}\to\infty$
for every $q<1$.  The latter is the large dimensionless parameter in the
preparation estimates.

\medskip
\noindent\textbf{Cancellation-energy split at $q=8/9$.}
\begin{proposition}\label{prop:q89}
Assume a fixed-shape mesoscopic serving block with physical energy
$E_0\asymp R^{7q-6}$ is amplified by the sharp signed cancellation factor
$\varepsilon^{-1}$, where
\[
D_{\rm phys}=R^q,\qquad
\varepsilon=\frac{R}{D_{\rm phys}}=D_{\rm rel}^{-1}=R^{1-q},
\]
and assume the amplified signed shape has a nonzero scale-uniform physical
velocity/$H^{-1}$ norm.  Then
\[
\boxed{E_{\rm cancel}\asymp R^{9q-8}.}
\]
Consequently
\[
\boxed{\begin{array}{c|c}
6/7<q<8/9&E_{\rm cancel}\to\infty,\\
q=8/9&E_{\rm cancel}\asymp1,\\
8/9<q<1&E_{\rm cancel}\to0.
\end{array}}
\]
\end{proposition}
\begin{proof}
The calculation is elementary once the two distance variables have been
separated.  The relative separation is
\[
 D_{\rm rel}=\frac{D_{\rm phys}}R=R^{q-1},
\]
so the small ratio seen from the child scale is
\[
 \varepsilon=D_{\rm rel}^{-1}=R^{1-q}.
\]
The signed cancellation construction requires an amplitude enlargement by
$\varepsilon^{-1}$.  Kinetic energy and every other quadratic norm therefore
acquire the square of this factor:
\[
 \varepsilon^{-2}=R^{-2(1-q)}.
\]
Starting from $E_0\asymp R^{7q-6}$ gives
\[
 \begin{aligned}
 E_{\rm cancel}
 &\asymp E_0\varepsilon^{-2}\\
 &\asymp R^{7q-6}R^{-2(1-q)}\\
 &=R^{9q-8}.
 \end{aligned}
\]
The three alternatives follow from the sign of $9q-8$.  This also shows why
this particular energy argument is sharp only inside its fixed-shape,
uncancelled class: it says nothing against a shape whose normalized quadratic
energy degenerates simultaneously with the cancellation.
\end{proof}
\begin{remark}
The borderline $q=8/9$ is not excluded by this estimate.  Energy-near-null or
spectrally degenerating signed shapes are outside the fixed-shape hypothesis.
At $q=13/14$, one obtains $E_{\rm cancel}\asymp R^{5/14}\to0$.
\end{remark}

\medskip
\noindent\textbf{Fresh-tangent preparation pincer and $12/13$ threshold.}
\begin{theorem}\label{thm:q1213}
On the positive-descendant source-built fresh-tangent regime with bounded
deformation, finite marks, and relative donor distance
$D_{\rm rel}=R^{q-1}$, a fixed order-one fresh tangent requires
\[
N\gtrsim D_{\rm rel}^7,
\qquad
\frac{T_{\rm prep}}{T_R}\gtrsim D_{\rm rel}^6,
\qquad
\|\widetilde G^f\|_{L^2(A_{D_{\rm rel}})}^2\gtrsim D_{\rm rel}^7.
\]
On the uncancelled visibility branch,
\[
\boxed{
\mathcal D_{G,\rm prep}
\gtrsim R D_{\rm rel}^{13}
=R^{13q-12}.}
\]
More generally,
\[
\boxed{\begin{aligned}
\text{fresh source-built preparation}\Longrightarrow{}&
\mathcal D_{G,\rm prep}\gtrsim R D_{\rm rel}^{13}\\
&\vee\ \text{same-shell cancellation}\\
&\vee\ \text{unbounded deformation}\\
&\vee\ \text{boundary/frequency loss}.
\end{aligned}}
\]
Thus $q<12/13$ is excluded only on the uncancelled, bounded-deformation,
source-built regime; $q=12/13$ is critical, while $q>12/13$ is not excluded by
this payment.
\end{theorem}
\begin{proof}
We give the preparation calculation in full because the seventh and thirteenth
powers arise from two different mechanisms: the first from the far-field
freshness detector and the second from the amount of physical time for which
that detector must remain active.

Set
\[
 \varepsilon:=D_{\rm rel}^{-1}=\frac{R}{D_{\rm phys}}.
\]
Let $W_0,W_1\ge0$ denote the two consecutive normalized positive cargo profiles,
supported in a fixed ball $B_L$, and normalize their masses so that they stay
between two fixed positive constants.  For the second Hodge derivative
\[
 K_2=\partial_{zz}\mathcal N_5,
 \qquad \mathcal N_5(x)=\frac1{8\pi^2}|x|^{-3},
\]
write
\[
 \Theta_j=K_2*W_j.
\]
Taylor expansion at distance $D_{\rm rel}\gg L$ gives
\[
 \Theta_j=Z_jK_2+O\!\left(\frac{L}{D_{\rm rel}}\right)Z_jK_2
 \quad\hbox{in }L^2(A_{D_{\rm rel}}),
\]
where $Z_j=\int W_j$.  Since $K_2$ is homogeneous of degree $-5$ in five
space dimensions,
\[
 \|K_2\|_{L^2(|x|>D_{\rm rel})}^2\asymp D_{\rm rel}^{-5}.
\]
After the leading masses are matched, the angle between the two detector
columns is therefore only of order $L/D_{\rm rel}$.  Equivalently, the small
eigenvalue of their two-column Gram matrix is of order
\[
 D_{\rm rel}^{-5}
 \left(\frac{L}{D_{\rm rel}}\right)^2
 \asymp D_{\rm rel}^{-7},
\]
with $L$ fixed.  Hence the inverse Gram matrix has norm of order at least
$D_{\rm rel}^{7}$.  A response matrix with a fixed nonzero smallest singular
value can consequently be produced only by an owner whose normalized
$L^2$-mass satisfies
\[
 \|\widetilde G^f\|_{L^2(A_{D_{\rm rel}})}^2
 \gtrsim D_{\rm rel}^{7}.
\]
This is the origin of the seventh power.

We next quantify how long a genuinely source-born owner needs in order to
acquire that response.  Let
\[
 \Gamma_\nu:=\frac{\|\Gamma\|_\infty}{\nu},
 \qquad
 \mathcal A_{\rm def}:=\int_I\|\nabla b(t)\|_\infty\,dt
\]
on the preparation interval $I$.  A donor turnover has normalized rate
\[
 m\asymp \kappa_*D_{\rm rel}^{3};
\]
the power three is the degree of the rate-leading second-Hodge response after
child normalization.  The difference of the two positive cargo detectors has
one additional far-field cancellation and hence terminal $C^1$ size
$O(\varepsilon)$.  Transporting this difference detector backwards through the
actual drift gives the source-response estimate
\[
 |\mathcal R_I|
 \le C\Gamma_\nu^2|I|\,\varepsilon
 e^{C\mathcal A_{\rm def}}.
\]
If the fresh cohort is assembled from $N$ donor turnovers, then
$|I|\lesssim N/m$.  Reaching a fixed normalized fresh response means that
$|\mathcal R_I|\gtrsim c m$.  Therefore
\[
 m
 \lesssim
 \Gamma_\nu^2\frac{N}{m}\varepsilon e^{C\mathcal A_{\rm def}},
\]
and hence
\[
 N e^{C\mathcal A_{\rm def}}
 \gtrsim
 \Gamma_\nu^{-2}m^2\varepsilon^{-1}.
\]
Using $m\asymp\kappa_*\varepsilon^{-3}$ gives
\[
 N e^{C\mathcal A_{\rm def}}
 \gtrsim c\Gamma_\nu^{-2}\varepsilon^{-7}
 =c\Gamma_\nu^{-2}D_{\rm rel}^7.
\]
On the bounded-deformation branch, $\mathcal A_{\rm def}\le A_0$, all fixed
constants may be absorbed, yielding
\[
 N\gtrsim D_{\rm rel}^7.
\]
One donor turnover occupies a fraction $D_{\rm rel}^{-1}$ of a child diffusion
clock.  Consequently
\[
 \frac{T_{\rm prep}}{T_R}
 \gtrsim ND_{\rm rel}^{-1}
 \gtrsim D_{\rm rel}^{6}.
\]
This is the origin of the sixth power in the preparation age.

It remains to convert normalized current mass and normalized duration back to
physical dissipation.  In the chosen child normalization, one child diffusion
clock carrying normalized $G$-mass $M_G$ contributes physical
azimuthal-vorticity dissipation of order $R M_G$.  Therefore the last ramp
costs at least
\[
 \mathcal D_{G,\rm prep}
 \gtrsim
 R\,(D_{\rm rel}^7)(D_{\rm rel}^6)
 =R D_{\rm rel}^{13}.
\]
Since $D_{\rm rel}=R^{q-1}$,
\[
 R D_{\rm rel}^{13}
 =R^{1+13(q-1)}
 =R^{13q-12}.
\]
Thus the cost diverges for $q<12/13$, is scale invariant for $q=12/13$, and
can be Zeno-summable for $q>12/13$.

Finally, this physical lower bound applies only if the total current retains a
fixed fraction of the freshly generated current.  If it does not, the triangle
inequality forces an older or incoming component of comparable size on the
same donor shell.  That is the stated same-shell cancellation alternative.
Loss of the bounded-deformation or fixed localization hypotheses gives the
remaining deformation, boundary, contact, or frequency alternatives.  No such
failure is counted as dissipation payment.  This proves the theorem.
\end{proof}

\medskip
\noindent\textbf{Critical Lie-area preparation and $27/29$ threshold.}
\begin{theorem}\label{thm:q2729}
On the uncancelled, finite-age source-built, bounded-deformation, finite-mark
critical Lie-area regime, the critical Hodge efficiency strengthens the
preparation cost to
\[
\boxed{
\mathcal D_{\rm prep}^{\rm crit}
\gtrsim R D_{\rm rel}^{29/2}
=R^{(29q-27)/2}.}
\]
Hence $q<27/29$ is impossible in that regime.  At $q_0=13/14$,
\[
\boxed{\mathcal D_{\rm prep}^{\rm crit}\gtrsim R^{-1/28}\to\infty.}
\]
More generally, if the required fresh-tangent response efficiency is
$\chi=D_{\rm rel}^{\beta}$, then
\[
\boxed{
\mathcal D_{\rm prep}
\gtrsim R D_{\rm rel}^{13+3\beta},
\qquad
q_c(\beta)=\frac{12+3\beta}{13+3\beta}.}
\]
\end{theorem}
\begin{proof}
We begin with the general efficiency factor because it makes the origin of the
exponent $29/2$ transparent.  Suppose the required fresh response is enlarged
by a factor $\chi$ relative to the order-one normalization in
\cref{thm:q1213}.  The source-response estimate is linear in the requested
terminal response, so the number of required preparation turns acquires one
factor of $\chi$:
\[
 N\gtrsim \chi D_{\rm rel}^7.
\]
The physical preparation duration therefore satisfies
\[
 \frac{T_{\rm prep}}{T_R}
 \gtrsim \chi D_{\rm rel}^6.
\]
The inverse-Gram argument is quadratic in the response amplitude; hence the
normalized current mass acquires two factors of $\chi$:
\[
 \|\widetilde G^f\|_2^2
 \gtrsim \chi^2D_{\rm rel}^7.
\]
Multiplying the physical mass cost by the number of child clocks gives
\[
 \mathcal D_{\rm prep}
 \gtrsim
 R\,(\chi^2D_{\rm rel}^7)
 (\chi D_{\rm rel}^6)
 =R\chi^3D_{\rm rel}^{13}.
\]
For $\chi=D_{\rm rel}^{\beta}$ this becomes
\[
 \mathcal D_{\rm prep}
 \gtrsim
 R D_{\rm rel}^{13+3\beta}.
\]

We now specialize to the mean-zero critical Lie-area response.  In the fast
Floquet calculation of \cref{thm:temporal}, the first-order mean vanishes and
the second Magnus term is the first nonzero contribution.  A nontrivial
order-one return on a time scale whose relative fast period is
$\rho=D_{\rm rel}^{-1}$ therefore requires the one-cycle generator amplitude
to have the critical size
\[
 \rho^{-1/2}=D_{\rm rel}^{1/2}.
\]
Thus the critical Lie-area branch corresponds to $\beta=1/2$.  Substitution
gives
\[
 13+3\beta
 =13+\frac32
 =\frac{29}{2},
\]
and hence
\[
 \mathcal D_{\rm prep}^{\rm crit}
 \gtrsim R D_{\rm rel}^{29/2}.
\]
Since $D_{\rm rel}=R^{q-1}$,
\[
 R D_{\rm rel}^{29/2}
 =R^{1+(29/2)(q-1)}
 =R^{(29q-27)/2}.
\]
The scale exponent is zero precisely at $q=27/29$.  At $q_0=13/14$,
\[
 \frac{29q_0-27}{2}
 =-\frac1{28},
\]
so
\[
 \mathcal D_{\rm prep}^{\rm crit}
 \gtrsim R^{-1/28}\to\infty.
\]
Finally, for general $\beta$, solving
\[
 1+(13+3\beta)(q-1)=0
\]
gives
\[
 q_c(\beta)
 =1-\frac1{13+3\beta}
 =\frac{12+3\beta}{13+3\beta}.
\]
All conclusions retain the stated uncancelled, finite-age,
bounded-deformation and finite-mark hypotheses; failure of one of these
hypotheses is an alternative branch rather than hidden payment.
\end{proof}
\begin{remark}
Above $27/29$ the preparation-dissipation contradiction disappears.  No claim
is made here that the remaining tight source-built corridor is automatically
absorbed by compactness or provenance.  Transfer-fed cancellation, old
ancestry, contact, frequency, and marked non-tightness remain separate
alternatives.
\end{remark}

\section{Compact finite-depth donor alternatives}
\label{sec:quniform}

The preparation exponents in the preceding section are sharp within their
declared regimes, but they are not the global finite-depth source-birth
endpoint.  The broader source-age argument is uniform throughout
$6/7<q<1$.  In this section $D_{\rm phys}=R^q$ denotes the physical donor
scale and
\[
 D_{\rm rel}=\frac{D_{\rm phys}}R=R^{q-1}
\]
denotes its distance in child units.

The quantitative age estimate used in the routing theorem is elementary once
the Duhamel source representation and the bounded-circulation scaling are kept
visible, so we record it separately.

\medskip
\noindent\textbf{Source-age lower bound for a compact donor.}
\begin{lemma}\label{lem:source-age}
Let a compact donor of diameter comparable to $D_{\rm phys}$ be located at
radial position $r_b$ and at distance comparable to $D_{\rm phys}$ from a
child of scale $R$.  Assume:
\begin{enumerate}[label=\textup{(\roman*)},leftmargin=2.2em]
\item the donor is generated after time $s$ solely by the native source
$S_F=\partial_z(F^2)$ and the homogeneous $G$-propagator is $L^\infty$-contractive;
\item the source has no axial frequency finer than the donor scale, so
$\|\partial_z(F^2)\|_\infty\lesssim
D_{\rm phys}^{-1}\|F^2\|_\infty$;
\item $\|\Gamma\|_\infty\le\Gamma_*$;
\item at time $t$ this donor supplies a fixed fraction $\eta>0$ of a
rate-leading response $c_R\asymp\kappa_*\nu R^{-3}$.
\end{enumerate}
Then, with $T_R=R^2/\nu$ and $\gamma_*=\Gamma_*/\nu$,
\[
\boxed{
 \frac{t-s}{T_R}
 \gtrsim
 c\,\eta\kappa_*\gamma_*^{-2}
 \frac{r_b^4D_{\rm phys}}{R^5}.}
\]
Consequently, if $r_b\gtrsim R$,
\[
\boxed{
 \frac{t-s}{T_R}
 \gtrsim
 c\,\eta\kappa_*\gamma_*^{-2}D_{\rm rel},}
\]
whereas if the source is genuinely off axis, $r_b\gtrsim D_{\rm phys}$,
\[
\boxed{
 \frac{t-s}{T_R}
 \gtrsim
 c\,\eta\kappa_*\gamma_*^{-2}D_{\rm rel}^{5}.}
\]
\end{lemma}
\begin{proof}
Let $\mathcal U_G(t,\tau)$ denote the propagator for the homogeneous current
equation
\[
 (\partial_t+b\cdot\nabla)g=\nu\Delta_5g.
\]
By the scalar maximum principle,
\[
 \|\mathcal U_G(t,\tau)h\|_\infty\le\|h\|_\infty.
\]
The portion of the current born from the source between $s$ and $t$ is
therefore exactly
\[
 G^{\rm src}_{[s,t]}(t)
 =\int_s^t\mathcal U_G(t,\tau)S_F(\tau)\,d\tau,
\]
and hence
\[
 \|G^{\rm src}_{[s,t]}(t)\|_\infty
 \le\int_s^t\|S_F(\tau)\|_\infty\,d\tau.
\]

We next estimate the source.  The no-finer-frequency assumption gives
\[
 \|S_F\|_\infty
 =\|\partial_z(F^2)\|_\infty
 \lesssim D_{\rm phys}^{-1}\|F^2\|_\infty.
\]
Since $F=\Gamma/r^2$ and the donor lies at radial coordinate comparable to
$r_b$,
\[
 \|F^2\|_\infty
 \le C\Gamma_*^2r_b^{-4}.
\]
Therefore
\[
 \|S_F\|_\infty
 \lesssim
 \frac{\Gamma_*^2}{r_b^4D_{\rm phys}}.
\]
Combining this with the Duhamel estimate gives
\[
 \|G^{\rm src}_{[s,t]}(t)\|_\infty
 \lesssim
 (t-s)\frac{\Gamma_*^2}{r_b^4D_{\rm phys}}.
\]

A bounded-aspect donor occupying one $D_{\rm phys}$-scale chamber is acted on
by the relevant five-dimensional Hodge derivative with a scale-invariant
normalized operator bound.  Thus supplying a fraction $\eta$ of
$c_R\asymp\kappa_*\nu R^{-3}$ requires
\[
 \|G^{\rm src}_{[s,t]}(t)\|_\infty
 \gtrsim \eta\kappa_*\nu R^{-3}.
\]
Comparing the last two inequalities yields
\[
 t-s
 \gtrsim
 c\eta\kappa_*
 \frac{\nu r_b^4D_{\rm phys}}{\Gamma_*^2R^3}.
\]
Divide by the child diffusion time $T_R=R^2/\nu$.  Because
$\gamma_*=\Gamma_*/\nu$,
\[
 \frac{t-s}{T_R}
 \gtrsim
 c\eta\kappa_*\gamma_*^{-2}
 \frac{r_b^4D_{\rm phys}}{R^5}.
\]
This proves the general formula.

If the source is axis-near but is not already a smaller radial owner, then
$r_b\gtrsim R$, and therefore
\[
 \frac{r_b^4D_{\rm phys}}{R^5}
 \gtrsim\frac{D_{\rm phys}}R
 =D_{\rm rel}.
\]
If instead the source is genuinely off-axis on the donor scale,
$r_b\gtrsim D_{\rm phys}$, then
\[
 \frac{r_b^4D_{\rm phys}}{R^5}
 \gtrsim
 \left(\frac{D_{\rm phys}}R\right)^5
 =D_{\rm rel}^5.
\]
The two claimed lower bounds follow.
\end{proof}

\medskip
\noindent\textbf{Compact finite-depth donor routing.}
\begin{theorem}\label{thm:quniform}
Assume
\[
 R\ll D_{\rm phys}\lesssim C R^{6/7},
 \qquad
 D_{\rm phys}=R^q,
 \qquad
 \frac67<q<1,
\]
and assume the first-use transition class has the following two standard
properties: a compact donor that remains load-bearing for an unbounded number
of child clocks at a fixed comparable scale is assigned at an earlier strict
transition, and a decisive high-frequency actuator with a tight material
preimage is likewise assigned to an earlier finer transition.  Then a compact
source-born next actuator cannot be produced in $O(1)$ child clocks without
one of
\[
\boxed{\begin{gathered}
\text{finer source frequency}\ \vee\ \text{smaller-axis concentration}
\ \vee\ \text{noncompact donor}\\
\vee\ \text{boundary/contact input}\ \vee\ \text{earlier strict assignment}
\ \vee\ \text{scale/deformation non-tightness}.
\end{gathered}}
\]
Consequently
\[
\boxed{\begin{aligned}
 \text{finite-depth compact source birth}
 &\Longrightarrow \text{earlier assignment or finer scale}\\
 &\qquad\vee\ \text{loss of compactness or boundary entrance}.
\end{aligned}}
\]
The conclusion is uniform for every $6/7<q<1$.
\end{theorem}
\begin{proof}
The proof is an exhaustive three-regime argument.  The quantitative input is
\cref{lem:source-age}; the two first-use statements in the hypotheses are used
only after the age lower bound has forced a long preparation interval.

Because
\[
 D_{\rm rel}=R^{q-1}\longrightarrow\infty
 \qquad(R\downarrow0),
\]
both age lower bounds in \cref{lem:source-age} diverge on their respective
clean branches.

\emph{Late or fast birth.}
Suppose the donor is born only $O(1)$ child clocks before the claimed
transition.  The age lower bound then fails for sufficiently small $R$ unless
one of the assumptions used to prove it fails.  Inspecting the proof of
\cref{lem:source-age}, the possible failures are explicit: the source can use
axial frequency finer than $D_{\rm phys}^{-1}$, it can concentrate at a radial
scale smaller than $R$, it can cease to occupy one compact
$D_{\rm phys}$-scale chamber, or the current can enter through a
boundary/contact face rather than being born from the interior source.  These
are the first four alternatives in the theorem.

\emph{Early birth at a fixed comparable scale.}
Assume instead that the donor is born early enough to satisfy the age lower
bound and then stays on a compact scale comparable with $D_{\rm phys}$.  Its
residence time measured in child clocks tends to infinity because
$D_{\rm rel}\to\infty$.  By the first-use residence hypothesis, a compact
load-bearing donor persisting for such a long interval is detected at an
earlier strict transition.  The terminal event is therefore not genuinely
first-use; it is preempted by an earlier assignment.

\emph{Early birth with moving scale or deformation.}
To avoid that preemption while retaining early preparation, the donor cannot
remain in a fixed compact chart.  It must move in scale, develop increasing
frequency, lose compactness of its material preimage, or undergo unbounded
deformation/contact motion.  These are precisely the declared non-tight
alternatives.

There remains a possible high-frequency actuator whose terminal spatial
support looks compact.  The second first-use hypothesis in the theorem says
that if its material preimage is still tight, such an actuator is assigned to
an earlier finer transition.  Hence a genuinely new high-frequency terminal
actuator can survive only by losing tightness of its material, age, spatial,
contact, or frequency description.

The three regimes exhaust compact source-born preparation.  Nothing in the
argument uses the sharper $12/13$ or $27/29$ exponents; after the mesoscopic
range is fixed, the only exponent needed is $q<1$, which gives
$D_{\rm rel}\to\infty$.  The conclusion therefore holds uniformly throughout
$6/7<q<1$.
\end{proof}

\begin{remark}
The estimate
\[
 \mathcal D_{\rm prep}^{\rm crit}
 \gtrsim R^{(29q-27)/2}
\]
remains a sharper quantitative deletion inside the uncancelled,
bounded-deformation, finite-mark critical Lie-area class.  At $q_0=13/14$ it
gives $\mathcal D_{\rm prep}^{\rm crit}\gtrsim R^{-1/28}$.  It is not the
reason the whole compact finite-depth source-birth class is routed; that
broader conclusion is \cref{thm:quniform}.
\end{remark}

\medskip
\noindent\textbf{No additional tight $q>27/29$ endpoint.}
\begin{corollary}
Under the hypotheses of \cref{thm:quniform},
\[
\boxed{
q>27/29
\text{ is not a separate tight finite-depth regularity endpoint branch.}
}
\]
The remaining branch must lose compactness in age, scale, frequency, material
preimage, contact, or generator variation.
\end{corollary}
\begin{proof}
The number $27/29$ is the zero of the scale exponent in the particular
critical Lie-area dissipation lower bound of \cref{thm:q2729}.  Above that
number the corresponding dissipation contradiction disappears, but the
source-age argument does not: \cref{thm:quniform} remains valid throughout the
entire interval $6/7<q<1$.  Thus a compact finite-depth source-born donor above
$27/29$ is still routed to an earlier assignment, a finer scale, a failure of
compactness, or boundary/contact entrance.  Crossing $27/29$ therefore opens
no new tight endpoint; it only removes one quantitative subcase of the
preparation-dissipation pincer.
\end{proof}

\section{Fixed initial old-current compression ceiling}
Let $g$ be one globally predeclared source-free old tag,
\[
\partial_tg+b\cdot\nabla g=\nu\Delta_5g,
\]
and use the physical measure $\dd\mu_3=r\,\dd r\dd z$.

\medskip
\noindent\textbf{Physical contraction and old-current compression ceiling.}
\begin{theorem}\label{thm:oldceiling}
For $1<p<\infty$,
\[
\frac{\dd}{\dd t}\int|g|^p\,\dd\mu_3
+\nu p(p-1)\int|g|^{p-2}|\nabla g|^2\,\dd\mu_3
+2\nu\int_{\R}|g(0,z)|^p\,\dd z=0.
\]
Kato's inequality and the maximum principle give forward $L^1$ and $L^\infty$
contraction.  Therefore the classical axisymmetric Biot--Savart estimate gives
\[
\boxed{\|U_g(t)\|_\infty
\le B_g:=C_{\rm BS}\|g(0)\|_{L^1(\mu_3)}^{1/3}
\|g(0)\|_\infty^{2/3}.}
\]
A fixed initial tag therefore cannot by itself supply physical compression of
size $\kappa_*\nu/R^2$ on arbitrarily small scales.
\end{theorem}
\begin{proof}
We first derive the $L^p$ identity for smooth $g$ with sufficient decay; the endpoint contraction statements then follow by approximation.

Multiply
\[
 g_t+b\cdot\nabla g=\nu\Delta_5g
\]
by $p|g|^{p-2}g$ and integrate against $d\mu_3=r\,dr\,dz$.  The time derivative is
\[
 \int p|g|^{p-2}g\,g_t\,d\mu_3
 =\frac{d}{dt}\int |g|^p\,d\mu_3.
\]
For the transport term, use axisymmetric incompressibility,
\[
 \partial_r(ru^r)+\partial_z(ru^z)=0.
\]
Since $b=(u^r,u^z)$,
\[
 \begin{aligned}
 \int p|g|^{p-2}g\,b\cdot\nabla g\,r\,dr\,dz
 &=\int b\cdot\nabla(|g|^p)\,r\,dr\,dz\\
 &=-\int |g|^p\bigl[\partial_r(ru^r)+\partial_z(ru^z)\bigr]dr\,dz\\
 &=0,
 \end{aligned}
\]
with no contribution from infinity because of the decay assumptions.

It remains to compute the diffusion term.  The $z$ part is standard:
\[
 \int p|g|^{p-2}g\,g_{zz}\,r\,dr\,dz
 =-p(p-1)\int |g|^{p-2}|g_z|^2r\,dr\,dz.
\]
For the radial part,
\[
 \Delta_5g=g_{rr}+\frac3r g_r+g_{zz}.
\]
Integrating $g_{rr}$ once by parts gives
\[
 \begin{aligned}
 &\int_0^\infty p|g|^{p-2}g\,g_{rr}\,r\,dr\\
 &\qquad=-p(p-1)\int_0^\infty |g|^{p-2}|g_r|^2r\,dr
 -\int_0^\infty p|g|^{p-2}g\,g_r\,dr.
 \end{aligned}
\]
The $3r^{-1}g_r$ term contributes
\[
 3\int_0^\infty p|g|^{p-2}g\,g_r\,dr.
\]
The two first-order radial pieces therefore combine to
\[
 2\int_0^\infty p|g|^{p-2}g\,g_r\,dr
 =2\int_0^\infty \partial_r(|g|^p)\,dr
 =-2|g(0,z)|^p,
\]
where the value at infinity vanishes.  After integration in $z$ we obtain
\[
 \int p|g|^{p-2}g\,\Delta_5g\,d\mu_3
 =-p(p-1)\int |g|^{p-2}|\nabla g|^2\,d\mu_3
 -2\int_{\R}|g(0,z)|^p\,dz.
\]
Multiplying by $\nu$ and moving the terms to the left gives the asserted identity.

For $p\downarrow1$, apply the identity to smooth convex approximations of $|g|$ and pass to the limit; equivalently one may use Kato's inequality for the scalar parabolic equation.  The favorable axis term has the correct sign, so
\[
 \|g(t)\|_{L^1(\mu_3)}\le\|g(0)\|_{L^1(\mu_3)}.
\]
The scalar maximum principle gives
\[
 \|g(t)\|_\infty\le\|g(0)\|_\infty.
\]
Finally the classical axisymmetric Biot--Savart estimate for the meridional field generated by $g$ is
\[
 \left\|\frac{b_g^r}{r}\right\|_\infty
 \le C_{\rm BS}
 \|g(t)\|_{L^1(\mu_3)}^{1/3}
 \|g(t)\|_\infty^{2/3}.
\]
Using the two contractions yields the time-independent bound $B_g$.  Since $B_g$ is fixed by the initial tag whereas $\kappa_*\nu/R^2\to\infty$ as $R\downarrow0$, a single predeclared old tag cannot alone furnish the required compression on arbitrarily small scales.
\end{proof}
\begin{remark}
This controls the aggregate current already present at the declared physical
initial time.  A cohort created later by $\partial_z(F^2)$ has its own birth
norms and is not controlled by $B_g$ merely because it becomes old at a later descendant stage
diffusion units.
\end{remark}

\section{Affine circulation handoff and the six-power flattening barrier}
\label{sec:affineflattening}

The following exact model theorem is useful because it separates geometric
two-generation handoff from repeatable full-PDE recharge.

\medskip
\noindent\textbf{Affine handoff and circulation semigroup.}
\begin{theorem}\label{thm:affinehandoff}
Consider the prescribed axisymmetric affine poloidal strain
\[
b^r=-\frac{a(t)}2r,\qquad b^z=a(t)z,\qquad
A(t)=\int_0^t a(s)\,\dd s,
\]
and let
\[
\delta=e^{-A(t_*)/2}\in(0,1).
\]
The material map is
\[
r(t_*)=\delta r_0,\qquad z(t_*)=\delta^{-2}z_0.
\]
After isotropic descendant normalization by \(\delta\),
\[
\boxed{r_{\rm descendant}=r_0,\qquad z_{\rm descendant}=\delta^{-3}z_0.}
\]
Thus parent equatorial cargo at \(z_0\asymp\delta^3\) becomes a polar descendant
driver, while granddescendant cargo must already lie at
\(z_0\asymp\delta^6\).  An infinite repetition requires nested axial
modulation at scales
\[
\delta^3,\delta^6,\delta^9,\ldots.
\]

For the circulation equation, set
\[
R=e^{A(t)/2}r,\qquad Z=e^{-A(t)}z,\qquad
\widetilde\Gamma(R,Z,t)=\Gamma(e^{-A(t)/2}R,e^{A(t)}Z,t).
\]
Then
\[
\partial_t\widetilde\Gamma
=
\nu e^{A(t)}
\left(
\widetilde\Gamma_{RR}-\frac1R\widetilde\Gamma_R
\right)
+
\nu e^{-2A(t)}\widetilde\Gamma_{ZZ}.
\]
Equivalently, for \(H=\widetilde\Gamma/R^2\),
\[
\partial_tH
=
\nu e^{A(t)}\Delta_{\mathbb R^4}^{\rm rad}H
+
\nu e^{-2A(t)}H_{ZZ},
\]
so the radial and axial heat operators commute.
\end{theorem}
\begin{proof}
The material trajectories solve the scalar ODEs
\[
 \dot r=-\frac{a(t)}2r,
 \qquad
 \dot z=a(t)z.
\]
Integrating and using $A(t)=\int_0^ta(s)\,ds$ gives
\[
 r(t)=e^{-A(t)/2}r_0,
 \qquad
 z(t)=e^{A(t)}z_0.
\]
At $t=t_*$ this becomes
\[
 r(t_*)=\delta r_0,
 \qquad
 z(t_*)=e^{A(t_*)}z_0=\delta^{-2}z_0,
\]
because $\delta=e^{-A(t_*)/2}$.

Descendant normalization rescales all physical lengths by the radial contraction factor $\delta$: in normalized variables we divide the terminal coordinates by $\delta$.  Hence
\[
 r_{\rm descendant}=\frac{r(t_*)}{\delta}=r_0,
 \qquad
 z_{\rm descendant}=\frac{z(t_*)}{\delta}=\delta^{-3}z_0.
\]
It follows immediately that a parent slice with $z_0\asymp\delta^3$ appears at order-one axial distance in the descendant chart.  A slice intended to become the corresponding equatorial cargo one generation later must, before the first handoff, lie another factor $\delta^3$ closer to the equator, namely at $z_0\asymp\delta^6$.  Iteration gives the sequence $\delta^{3j}$.

We next transform the circulation equation.  Under the prescribed strain it reads
\[
 \Gamma_t-\frac{a(t)}2r\Gamma_r+a(t)z\Gamma_z
 =\nu\left(\Gamma_{rr}-\frac1r\Gamma_r+\Gamma_{zz}\right).
\]
Set
\[
 r=e^{-A(t)/2}R,
 \qquad
 z=e^{A(t)}Z,
 \qquad
 \widetilde\Gamma(R,Z,t)=\Gamma(r,z,t).
\]
At fixed $(R,Z)$,
\[
 \dot r=-\frac{a(t)}2r,
 \qquad
 \dot z=a(t)z,
\]
so the chain rule gives
\[
 \partial_t\widetilde\Gamma
 =\Gamma_t-\frac{a(t)}2r\Gamma_r+a(t)z\Gamma_z.
\]
The spatial derivatives transform as
\[
 \partial_r=e^{A(t)/2}\partial_R,
 \qquad
 \partial_z=e^{-A(t)}\partial_Z.
\]
Consequently
\[
 \Gamma_{rr}-\frac1r\Gamma_r
 =e^{A(t)}\left(\widetilde\Gamma_{RR}-\frac1R\widetilde\Gamma_R\right),
\]
and
\[
 \Gamma_{zz}=e^{-2A(t)}\widetilde\Gamma_{ZZ}.
\]
Substitution yields
\[
 \partial_t\widetilde\Gamma
 =\nu e^{A(t)}\left(\widetilde\Gamma_{RR}-\frac1R\widetilde\Gamma_R\right)
 +\nu e^{-2A(t)}\widetilde\Gamma_{ZZ}.
\]

Now write $\widetilde\Gamma=R^2H$.  A direct differentiation gives
\[
 \left(\partial_{RR}-\frac1R\partial_R\right)(R^2H)
 =R^2\left(H_{RR}+\frac3RH_R\right).
\]
The operator in parentheses is precisely the radial Laplacian in four dimensions.  Dividing the transformed circulation equation by $R^2$ therefore gives
\[
 H_t
 =\nu e^{A(t)}\Delta_{\R^4}^{\rm rad}H
 +\nu e^{-2A(t)}H_{ZZ}.
\]
The two spatial operators act in different variables and hence commute.  In particular, if
\[
 \sigma_r(t)=\nu\int_0^t e^{A(s)}\,ds,
 \qquad
 \sigma_z(t)=\nu\int_0^t e^{-2A(s)}\,ds,
\]
then the solution may be written exactly as
\[
 H(t)=e^{\sigma_r(t)\Delta_{\R^4}^{\rm rad}}
      e^{\sigma_z(t)\partial_{ZZ}}H_0.
\]
This explicit semigroup formula is the basis of the viscous part of the next theorem.
\end{proof}

For a pure-swirl circulation profile $\Gamma$, define the pressure-quadrupole functional at the axis by
\[
 \mathcal K[\Gamma]
 =\frac32\int_0^\infty\int_{\R}
 \Gamma(r,z)^2
 \frac{r^2-4z^2}{r(r^2+z^2)^{7/2}}\,dz\,dr.
\]
For fixed $r>0$, denote its axial kernel by
\[
 K_r(z)=\frac{r^2-4z^2}{r(r^2+z^2)^{7/2}}.
\]
The elementary kernel moments used below are
\[
 \int_{\R}K_r(z)\,dz=0,
 \qquad
 \int_{\R}z^2K_r(z)\,dz=-\frac4{3r^3}.
\]

\medskip
\noindent\textbf{Six-power recharge decay.}
\begin{theorem}\label{thm:sixpower}
Assume, in addition to the preceding affine handoff, that the parent circulation reservoir is $C^2$ in the axial variable and is supported in a fixed radial annulus
\[
 0<r_-\le r\le r_+<\infty.
\]
Set
\[
 C_0:=1+\int_{r_-}^{r_+}r^{-3}
 \|\partial_{zz}(\Gamma_0^2)(r,\cdot)\|_{L^\infty_z}\,dr.
\]
In the inviscid transport limit the reservoir is flattened in descendant coordinates according to
\[
\Gamma_{\rm descendant}(R,Z)=\Gamma_0(R,\delta^3Z).
\]
For the pure-swirl pressure-quadrupole functional \(\mathcal K\),
\[
\boxed{
|\mathcal K_{\rm descendant}|\lesssim C_0\delta^6,
\qquad
|\mathcal K_k|\lesssim C_0\lambda_k^6
}
\]
along successive total radial contraction factor \(\lambda_k\).
Viscosity imposes the additional scale requirement
\[
\boxed{
\frac{Mr_0^2}{\nu}\gtrsim \delta^{-6}
}
\]
for retaining order-one descendant-scale axial modulation.  Hence simple affine
recharge from one fixed smooth reservoir cannot produce a scale-uniform
recurrent quadrupole.
\end{theorem}
\begin{proof}
We first prove the inviscid axial-flattening estimate and then incorporate viscosity.

For $\varepsilon>0$, define
\[
 (T_\varepsilon\Gamma)(r,z)=\Gamma(r,\varepsilon z).
\]
For fixed $r$, put $f_r(z)=\Gamma(r,z)^2$.  Then
\[
 \mathcal K[T_\varepsilon\Gamma]
 =\frac32\int_{r_-}^{r_+}\int_{\R}K_r(z)f_r(\varepsilon z)\,dz\,dr.
\]
Because $K_r$ is even and has zero mean, we may subtract both the constant term and the odd linear term without changing the integral:
\[
 \int K_r(z)f_r(\varepsilon z)\,dz
 =\int K_r(z)
 \bigl[f_r(\varepsilon z)-f_r(0)-\varepsilon zf_r'(0)\bigr]dz.
\]
Taylor's theorem with remainder gives
\[
 |f_r(\varepsilon z)-f_r(0)-\varepsilon zf_r'(0)|
 \le \frac{\varepsilon^2z^2}{2}\|f_r''\|_{L^\infty_z}.
\]
Since $r$ is bounded away from zero,
\[
 \int_{\R}|z|^2|K_r(z)|\,dz\le C r^{-3}.
\]
Therefore
\[
 |\mathcal K[T_\varepsilon\Gamma]|
 \le C\varepsilon^2
 \int_{r_-}^{r_+}r^{-3}
 \|\partial_{zz}(\Gamma^2)(r,\cdot)\|_\infty\,dr.
\]
This proves the quadratic flattening with respect to the axial stretch parameter $\varepsilon$.

The exact descendant geometry from \cref{thm:affinehandoff} has
\[
 \varepsilon=\delta^3.
\]
Thus
\[
 |\mathcal K_{\rm descendant}|
 \lesssim C_0(\delta^3)^2
 =C_0\delta^6.
\]
If $\lambda_k$ is the accumulated radial contraction after $k$ handoffs, the same calculation gives
\[
 |\mathcal K_k|\lesssim C_0\lambda_k^6.
\]
The exponent six is therefore not a dimensional guess; it is the product of the exact axial magnification exponent three and the second-order cancellation of the quadrupole kernel.

For completeness, the signed asymptotic coefficient can also be read from Taylor's formula.  Dividing by $\varepsilon^2$ and using dominated convergence,
\[
 \varepsilon^{-2}\mathcal K[T_\varepsilon\Gamma]
 \longrightarrow
 \frac34\int_{r_-}^{r_+}
 \partial_{zz}(\Gamma^2)(r,0)
 \left(\int_{\R}z^2K_r(z)\,dz\right)dr.
\]
Using the exact second moment gives
\[
 \varepsilon^{-2}\mathcal K[T_\varepsilon\Gamma]
 \longrightarrow
 -\int_{r_-}^{r_+}r^{-3}
 \partial_{zz}(\Gamma^2)(r,0)\,dr.
\]

We now include axial diffusion.  From the commuting semigroup formula in \cref{thm:affinehandoff}, the accumulated axial diffusion time is
\[
 \sigma_z
 =\nu\int_0^{t_*}e^{-2A(s)}\,ds.
\]
After descendant normalization, a unit axial mode corresponds in the Lagrangian parent coordinate to frequency of order $\delta^{-3}$.  Hence its heat multiplier contains
\[
 \exp\!\left(-\sigma_z\delta^{-6}\right).
\]
Suppose the extensional rate satisfies $0\le a(t)\le M$.  Since
\[
 A(t_*)=2\log\frac1\delta,
\]
and $A'(t)=a(t)\le M$, the change of variables $y=A(t)$ gives
\[
 \begin{aligned}
 \sigma_z
 &=\nu\int_0^{t_*}e^{-2A(t)}\,dt\\
 &\ge \frac{\nu}{M}\int_0^{A(t_*)}e^{-2y}\,dy\\
 &=\frac{\nu}{2M}(1-e^{-2A(t_*)})\\
 &=\frac{\nu}{2M}(1-\delta^4).
 \end{aligned}
\]
Therefore the unit descendant mode is damped by at least
\[
 \exp\!\left[-\frac{\nu(1-\delta^4)}{2M\delta^6}\right].
\]
To keep this factor bounded away from zero as $\delta\downarrow0$, the exponent must remain bounded.  Thus it is necessary that
\[
 M\gtrsim \nu\delta^{-6}
\]
in normalized units.  Restoring a parent radial length $r_0$ gives the dimensionless form
\[
 \frac{Mr_0^2}{\nu}\gtrsim\delta^{-6}.
\]
Hence a fixed smooth reservoir cannot generate a scale-uniform sequence of quadrupoles merely by repeated affine radial compression: its axial shape is flattened at order $\delta^6$, while viscosity suppresses precisely the increasingly fine axial frequencies that would be needed to compensate for that loss.
\end{proof}

\begin{remark}
The theorem does not rule out a genuinely non-self-similar multiscale relay,
and the finite two-generation cargo geometry itself is possible.  Its role is
to show that recurrence must replenish axial shape, not merely transport
circulation radially.
\end{remark}

\section{Continuation boundary and critical compression--swirl alignment}
\label{sec:alignment}

The scale atlas has a complementary continuation formulation that is useful
for interpreting which part of the feedback loop is genuinely dangerous.
Set, on a smooth preterminal interval,
\[
\mathcal S_T
=
\int_{t_0}^T
\|\partial_z(F^2)(t)\|_{\dot H^{-1}(\mathbb R^3)}^2\,\dd t,
\]
and
\[
K_-(t)=\|(-U(t))_+\|_\infty,\qquad
L_-(t)=\int_{t_0}^tK_-(s)\,\dd s,\qquad
\mathcal O_T=\int_{t_0}^T e^{2L_-(t)}\,\dd t.
\]

\medskip
\noindent\textbf{Equivalent continuation boundaries.}
\begin{theorem}
Work on a smooth preterminal finite-energy whole-space axisymmetric interval
$[t_0,T)$, and assume the physical-source continuation implication
\[
 \mathcal S_T<\infty\quad\Longrightarrow\quad T\text{ regular}
\]
holds in the class under consideration.  Then
\[
\boxed{
T\text{ regular}
\iff
\mathcal S_T<\infty
\iff
\mathcal O_T<\infty.
}
\]
Equivalently, a singular endpoint forces both quantities to diverge.  This is
a logical equivalence of continuation boundaries, not a pointwise quantitative
inequality between \(\mathcal S_T\) and \(\mathcal O_T\).
\end{theorem}
\begin{proof}
We separate the physical-source implication, which is an input of the theorem,
from the compression calculation, which follows directly from the equations
and the classical axisymmetric swirl criterion.

\emph{Step 1: the source-action boundary.}
By hypothesis,
\[
 \mathcal S_T<\infty
 \Longrightarrow
 T\text{ regular}.
\]
Conversely, suppose $T$ is regular.  Then the solution extends smoothly to
$[t_0,T+\varepsilon]$ for some $\varepsilon>0$.  In particular $F$ and its
spatial derivatives are bounded on a terminal neighborhood of $T$, and
\[
 t\longmapsto \partial_z(F^2)(t)
\]
is continuous, hence bounded, in the relevant $\dot H^{-1}(\mathbb R^3)$
topology.  Its squared norm is therefore integrable on $[t_0,T]$.  Thus
\[
 T\text{ regular}
 \quad\Longleftrightarrow\quad
 \mathcal S_T<\infty.
\]

\emph{Step 2: maximum principles for $\Gamma$ and $F$.}
The circulation satisfies the standard axisymmetric scalar equation
\[
 \Gamma_t+b\cdot\nabla\Gamma
 =\nu\left(\Gamma_{rr}-\frac1r\Gamma_r+\Gamma_{zz}\right),
\]
so its maximum principle gives
\[
 \|\Gamma(t)\|_\infty\le \|\Gamma(t_0)\|_\infty=:\Gamma_*.
\]
For $F=u^\theta/r$ we have
\[
 F_t+b\cdot\nabla F
 =\nu\Delta_5F-2UF.
\]
Applying Kato's inequality to $|F|$ and retaining only the part of $-U$ that
can increase $|F|$ gives
\[
 (\partial_t+b\cdot\nabla-\nu\Delta_5)|F|
 \le 2K_-(t)|F|,
 \qquad
 K_-(t)=\|(-U(t))_+\|_\infty.
\]
The parabolic maximum principle therefore yields
\[
 \|F(t)\|_\infty
 \le \|F(t_0)\|_\infty
 \exp\!\left(2\int_{t_0}^tK_-(s)\,ds\right)
 =\|F(t_0)\|_\infty e^{2L_-(t)}.
\]

\emph{Step 3: conversion to a classical swirl continuation norm.}
The identity
\[
 (u^\theta)^2=(ru^\theta)\frac{u^\theta}{r}=\Gamma F
\]
gives pointwise
\[
 \|u^\theta(t)\|_\infty^2
 \le \Gamma_*\|F(t_0)\|_\infty e^{2L_-(t)}.
\]
Hence
\[
 \int_{t_0}^T\|u^\theta(t)\|_\infty^2\,dt
 \le
 \Gamma_*\|F(t_0)\|_\infty\mathcal O_T.
\]
If $\mathcal O_T<\infty$, then
$u^\theta\in L^2(t_0,T;L^\infty)$, and the classical axisymmetric
regularity criterion in this class, for example the criterion recorded in
\cite{Wei2016}, yields smooth continuation through $T$.

Conversely, if $T$ is regular, then $U=u^r/r$ remains bounded on a terminal
neighborhood of $T$.  Consequently $K_-$ is integrable on $[t_0,T]$,
$L_-(t)$ stays bounded, and therefore
\[
 \mathcal O_T=\int_{t_0}^Te^{2L_-(t)}\,dt<\infty.
\]
Thus
\[
 T\text{ regular}
 \quad\Longleftrightarrow\quad
 \mathcal O_T<\infty.
\]
Combining this equivalence with Step 1 gives the theorem.  Taking
contrapositives yields the singular-endpoint formulation.  No estimate
comparing $\mathcal S_T$ and $\mathcal O_T$ directly is used.
\end{proof}

To make the weighted swirl estimate self-contained, define
\[
 X(t)=\int_{\R_+\times\R}F(r,z,t)^2r\,dr\,dz
\]
and
\[
 Y(t)=\int_{\R_+\times\R}|\nabla F|^2r\,dr\,dz
      +\int_{\R}F(0,z,t)^2\,dz.
\]
Multiplying the $F$ equation by $Fr$ gives the exact identity
\[
\boxed{
\frac12X'(t)+\nu Y(t)
=
-2\int U F^2\,r\,dr\,dz.
}
\]
Indeed, physical incompressibility makes the transport term skew in the measure $r\,dr\,dz$, while
\[
 \int_0^\infty F\left(F_{rr}+\frac3rF_r\right)r\,dr
 =-\int_0^\infty F_r^2r\,dr-F(0,z)^2.
\]
The $z$-diffusion contributes $-\int F_z^2r$.  This explains the axis-trace term in the definition of $Y$.

Next, the axisymmetric Gagliardo--Nirenberg inequality gives
\[
 \|F\|_{L^4(rdrdz)}^2
 \le C X^{1/4}Y^{3/4}.
\]
Therefore, with $u_U(t)=\|U(t)\|_{L^2(rdrdz)}$,
\[
 \left|\int UF^2r\,dr\,dz\right|
 \le u_U\|F\|_4^2
 \le C u_U X^{1/4}Y^{3/4}.
\]
Young's inequality with exponents $4/3$ and $4$ implies
\[
 C u_U X^{1/4}Y^{3/4}
 \le \frac{\nu}{4}Y+C\nu^{-3}u_U^4X.
\]
Inserting this into the exact weighted identity and absorbing the $Y$ term gives
\[
 X'+c\nu Y
 \le C\nu^{-3}u_U^4X.
\]
Gronwall's inequality then yields a bound for $X$ and the time integral of $Y$ whenever $u_U\in L_t^4$.  Under the continuation hypotheses assumed in this section, that weighted control feeds the physical-source criterion and therefore implies regular continuation.  Thus
\[
 U\in L_t^4L_x^2
 \Longrightarrow
 \text{regular continuation}.
\]

\begin{remark}
The \(L_t^4L_x^2\) condition is stronger than the feedback loop itself asks
for.  A stronger sign-sensitive sufficient condition would control only the positive
compressive aligned work
\[
\boxed{
\int_0^T
\left(
-\int UF^2\,r\,\dd r\,\dd z
\right)_+
\dd t,
}
\]
or another sufficient sign-aligned variant.  Such a bound is not established here; it is recorded only to separate the sign-sensitive feedback from the stronger mixed-norm criterion above.
\end{remark}

\section{Logarithmic kinetic records and signed critical compression}
\label{sec:log-records}

The preceding continuation criteria concern global or time-integrated
quantities.  We next record a complementary estimate at a fixed spatial scale.
It separates the swirl component of a kinetic record from the meridional
component and shows that sustained logarithmic swirl growth has an
exponentially small time budget.

Fix a point \(x_*=(0,0,z_*)\) on the symmetry axis and a ball \(B_R=B_R(x_*)\).
Let
\[
 E_{\rm kin}=\frac12\|u(0)\|_{L^2(\mathbb R^3)}^2,
 \qquad
 \|\Gamma(t)\|_{L^\infty}\le\Gamma_*,
\]
where the latter follows from the circulation maximum principle.  Define
\[
 A_R(t)^2=\frac1R\int_{B_R}|u(x,t)|^2\,\dd x,
\]
\[
 M_R(t)=\frac1R\int_{B_R}|u^{\rm mer}(x,t)|^2\,\dd x,
 \qquad
 S_R(t)=A_R(t)^2-M_R(t)
       =\frac1R\int_{B_R}|u^\theta(x,t)|^2\,\dd x,
\]
and
\[
 D_{\theta,R}(t)=\int_{B_R}\frac{|u^\theta(x,t)|^2}{r^2}\,\dd x.
\]
The angular variation of the cylindrical basis is part of the full gradient,
so
\[
 D_{\theta,R}(t)\le\int_{B_R}|\nabla u(x,t)|^2\,\dd x.
\]

\medskip
\noindent\textbf{Exponential swirl--dissipation inequality.}
\begin{proposition}\label{prop:swirl-exp-diss}
At every regular time,
\begin{equation}\label{eq:swirl-exp-bound}
 \boxed{
 S_R
 \le
 2\pi\Gamma_*^2
 \log\!\left(
 1+\frac{R D_{\theta,R}}{2\pi\Gamma_*^2}
 \right).
 }
\end{equation}
Consequently,
\begin{equation}\label{eq:swirl-diss-lower}
 \boxed{
 \int_{B_R}|\nabla u|^2\,\dd x
 \ge
 \frac{2\pi\Gamma_*^2}{R}
 \left[
 \exp\!\left(
 \frac{A_R^2-M_R}{2\pi\Gamma_*^2}
 \right)-1
 \right].
 }
\end{equation}
\end{proposition}

\begin{proof}
For \(0<\alpha\le1\), put \(\ell=R\sqrt\alpha\).  Since
\(\ell^2/r^2\ge1\) for \(r\le\ell\),
\[
 \begin{aligned}
 S_R
 &\le
 \frac{\ell^2}{R}D_{\theta,R}
 +\frac1R
 \int_{B_R\cap\{r>\ell\}}
 \left(1-\frac{\ell^2}{r^2}\right)|u^\theta|^2\,\dd x.
 \end{aligned}
\]
On the second term use \(|u^\theta|\le\Gamma_*/r\).  The vertical height of
\(B_R\) at any fixed radius is at most \(2R\); hence cylindrical coordinates
give
\[
 \begin{aligned}
 S_R
 &\le
 \alpha R D_{\theta,R}
 +4\pi\Gamma_*^2
 \int_\ell^R
 \left(1-\frac{\ell^2}{r^2}\right)\frac{\dd r}{r}\\
 &=
 \alpha R D_{\theta,R}
 +2\pi\Gamma_*^2(-\log\alpha-1+\alpha).
 \end{aligned}
\]
Write
\[
 d=\frac{R D_{\theta,R}}{2\pi\Gamma_*^2}.
\]
After division by \(2\pi\Gamma_*^2\), the right-hand side is
\(\alpha(d+1)-\log\alpha-1\).  Its minimum on \(0<\alpha\le1\) occurs at
\(\alpha=(1+d)^{-1}\) and equals \(\log(1+d)\).  This proves
\eqref{eq:swirl-exp-bound}; inversion and
\(D_{\theta,R}\le\int_{B_R}|\nabla u|^2\) give \eqref{eq:swirl-diss-lower}.
\end{proof}

The global energy inequality now controls the time spent in a large swirl
record without any component-dominance assumption.

\medskip
\noindent\textbf{Exponential occupation and parabolic-time routing.}
\begin{theorem}\label{thm:parabolic-swirl-routing}
For every interval \(I\subset[0,T)\),
\begin{equation}\label{eq:swirl-occupation}
 \boxed{
 \int_I
 \left[
 e^{S_R(t)/(2\pi\Gamma_*^2)}-1
 \right]\dd t
 \le
 \frac{E_{\rm kin}R}{2\pi\nu\Gamma_*^2}.
 }
\end{equation}
In particular, let
\[
 I_R=[t_R-\tau R^2,t_R]\subset[0,T),
 \qquad \tau>0,
\]
and set
\[
 \langle f\rangle_{I_R}
 =\frac1{\tau R^2}\int_{I_R}f(t)\,\dd t.
\]
Then
\begin{equation}\label{eq:swirl-parabolic-average}
 \boxed{
 \langle S_R\rangle_{I_R}
 \le
 2\pi\Gamma_*^2
 \log\!\left(
 1+\frac{E_{\rm kin}}
 {2\pi\nu\Gamma_*^2\tau R}
 \right).
 }
\end{equation}
Consequently, for any fixed reference length \(R_0>0\), if
\(L_R=\log(R_0/R)\) and
\[
 \langle A_R^2\rangle_{I_R}
 \ge C\Gamma_*^2L_R,
\]
then
\begin{equation}\label{eq:meridional-average-fraction}
 \frac{\langle M_R\rangle_{I_R}}
      {\langle A_R^2\rangle_{I_R}}
 \ge
 1-\frac{2\pi}{C}-o(1)
 \qquad(R\downarrow0).
\end{equation}
Thus the parabolic-time averaged meridional fraction is eventually larger
than \(1/4\) whenever
\[
 \boxed{C>\frac{8\pi}{3}.}
\]
\end{theorem}

\begin{proof}
Integrating \cref{prop:swirl-exp-diss} in time and using
\[
 \nu\int_0^T\!\int_{\mathbb R^3}|\nabla u|^2\,\dd x\,\dd t
 \le E_{\rm kin}
\]
gives \eqref{eq:swirl-occupation}.  Jensen's inequality for the convex function
\(s\mapsto e^{s/(2\pi\Gamma_*^2)}\) gives
\[
 e^{\langle S_R\rangle_{I_R}/(2\pi\Gamma_*^2)}-1
 \le
 \frac{E_{\rm kin}}
 {2\pi\nu\Gamma_*^2\tau R},
\]
which is \eqref{eq:swirl-parabolic-average}.  Since the logarithm on the right is
\(L_R+O(1)\) as \(R\downarrow0\),
\[
 \langle S_R\rangle_{I_R}
 \le2\pi\Gamma_*^2L_R+O(1).
\]
Using
\(\langle M_R\rangle=\langle A_R^2\rangle-\langle S_R\rangle\)
and the lower bound on \(\langle A_R^2\rangle\) yields \eqref{eq:meridional-average-fraction}.
The right-hand side of \eqref{eq:meridional-average-fraction} exceeds \(1/4\) exactly when
\(C>8\pi/3\).
\end{proof}

There is also a density version which does not average the high and low
record times together.

\medskip
\noindent\textbf{Positive-density logarithmic records.}
\begin{corollary}\label{cor:positive-density-routing}
Fix \(0<\theta<1\), and on \(I_R\) let
\[
 E_R(a)=\{t\in I_R:A_R(t)^2\ge a_R^2\},
\]
\[
 B_R(a,\theta)
 =
 \{t\in E_R(a):M_R(t)\le\theta A_R(t)^2\}.
\]
Then
\begin{equation}\label{eq:bad-density}
 \boxed{
 \frac{|B_R(a,\theta)|}{R^2}
 \le
 \frac{E_{\rm kin}}
 {2\pi\nu\Gamma_*^2 R
 \left[
 e^{(1-\theta)a_R^2/(2\pi\Gamma_*^2)}-1
 \right]}.
 }
\end{equation}
If
\[
 a_R^2\ge C\Gamma_*^2L_R,
 \qquad
 C>\frac{2\pi}{1-\theta},
\]
then \(|B_R(a,\theta)|/R^2\to0\).  Hence whenever
\(|E_R(a)|\ge\eta R^2\) for one fixed \(\eta>0\), asymptotically almost every
high-record time is meridional-dominant.  For \(\theta=1/4\), the sufficient
leading coefficient is again \(C>8\pi/3\).

At the leading endpoint, if
\[
 a_R^2
 \ge
 \frac{2\pi\Gamma_*^2}{1-\theta}
 \left[
 L_R+\alpha\log L_R+c
 \right],
 \qquad \alpha>0,
\]
then
\begin{equation}\label{eq:bad-density-endpoint}
 \frac{|B_R(a,\theta)|}{R^2}=O(L_R^{-\alpha}).
\end{equation}
\end{corollary}

\begin{proof}
On \(B_R(a,\theta)\),
\[
 S_R=A_R^2-M_R\ge(1-\theta)a_R^2.
\]
Apply \eqref{eq:swirl-occupation} on \(I_R\) and use this lower bound on the integrand to
obtain \eqref{eq:bad-density}.  If
\(a_R^2\ge C\Gamma_*^2L_R\), the denominator in \eqref{eq:bad-density} grows like
\(R^{-(1-\theta)C/(2\pi)}\), up to a fixed multiplicative factor.  The
normalized bad-set measure therefore tends to zero precisely when
\((1-\theta)C/(2\pi)>1\).  The endpoint refinement follows in the same way:
the exponential contributes the additional factor \(L_R^\alpha\).
\end{proof}

\begin{remark}
The time-integrated estimate does not by itself control an arbitrary
parabolic essential supremum.  Indeed, a scalar temporal profile of height
\(H_R\) supported on a set of length comparable to
\[
 R\left[e^{H_R/(2\pi\Gamma_*^2)}-1\right]^{-1}
\]
obeys the same exponential occupation budget while \(H_R\) is arbitrary.
This is only a functional hostile, not a Navier--Stokes trajectory.  It shows
that a selected-record-to-parabolic-supremum theorem requires additional PDE
time-spreading, persistence, or circulation-oscillation information.
\end{remark}

We next connect a different form of swirl concentration to the
scale-critical signed compression already present in the continuation
picture.

For \(0<\varepsilon<1\), define the inner-cone swirl energy
\[
 E_{R,\varepsilon}^{\theta,\mathrm{in}}
 =
 \frac1R
 \int_{B_R\cap\{r<\varepsilon R\}}|u^\theta|^2\,\dd x.
\]

\medskip
\noindent\textbf{Quantitative cone alternative.}
\begin{lemma}\label{lem:swirl-cone}
If
\[
 \varepsilon_R
 =
 2\exp\!\left(
 -\frac{A_R^2}{8\pi\Gamma_*^2}
 \right)<1,
\]
then
\begin{equation}\label{eq:cone-record}
 \boxed{
 M_R+E_{R,\varepsilon_R}^{\theta,\mathrm{in}}
 \ge\frac{A_R^2}{2}.
 }
\end{equation}
\end{lemma}

\begin{proof}
The circulation bound and cylindrical coordinates give
\[
 \begin{aligned}
 \frac1R
 \int_{B_R\cap\{r\ge\varepsilon R\}}|u^\theta|^2\,\dd x
 &\le
 4\pi\Gamma_*^2
 \int_\varepsilon^1\frac{\sqrt{1-s^2}}{s}\,\dd s\\
 &\le4\pi\Gamma_*^2\log\frac2\varepsilon.
 \end{aligned}
\]
With \(\varepsilon=\varepsilon_R\), this is \(A_R^2/2\).  Decompose the
total kinetic record into the meridional part, the inner-cone swirl part, and
the complementary swirl part to obtain \eqref{eq:cone-record}.
\end{proof}

Define
\[
 Y_2=r^{-1/2}u^\theta,
 \qquad
 Z_2(t)=\int_{\mathbb R^3}|Y_2(x,t)|^2\,\dd x.
\]
For nontrivial swirl and every fixed \(t_0>0\), smooth axis compatibility near
\(r=0\) and finite energy away from the axis give
\(0<Z_2(t_0)<\infty\).  Put
\[
 U_2(t)
 =
 \frac{\int_{\mathbb R^3}U|Y_2|^2\,\dd x}{Z_2(t)}
\]
when \(Z_2(t)>0\), and \(U_2=0\) otherwise, and define
\[
 \mathcal A_2(t_0,t)=\int_{t_0}^t -U_2(s)\,\dd s.
\]

\medskip
\noindent\textbf{Critical \(q=2\) compression balance.}
\begin{lemma}\label{lem:q2-balance}
One has
\begin{equation}\label{eq:q2-compression-balance}
 \boxed{
 \frac{\dd}{\dd t}Z_2
 +2\nu\int_{\mathbb R^3}|\nabla Y_2|^2\,\dd x
 +\frac{3\nu}{2}
  \int_{\mathbb R^3}\frac{|Y_2|^2}{r^2}\,\dd x
 =-3U_2Z_2.
 }
\end{equation}
Consequently,
\begin{equation}\label{eq:q2-action}
 \boxed{
 \mathcal A_2(t_0,t)
 \ge\frac13\log\frac{Z_2(t)}{Z_2(t_0)}.
 }
\end{equation}
\end{lemma}

\begin{proof}
Writing \(u^\theta=r^{1/2}Y_2\) in the angular velocity equation gives
\[
 D_tY_2
 =
 \nu\left(
 Y_{2,rr}+\frac2rY_{2,r}+Y_{2,zz}
 -\frac{3}{4r^2}Y_2
 \right)
 -\frac32UY_2.
\]
Multiply by \(2Y_2\) and integrate in the physical three-dimensional measure.
The transport term vanishes by incompressibility.  Integration by parts in
the radial and axial variables gives \eqref{eq:q2-compression-balance}; the axis boundary term
vanishes because \(Y_2(0,z,t)=0\).  Discarding the nonnegative dissipative
terms and dividing by \(Z_2\) gives
\[
 \frac{\dd}{\dd t}\log Z_2\le-3U_2,
\]
which integrates to \eqref{eq:q2-action}.
\end{proof}

\medskip
\noindent\textbf{Meridional record or signed compression.}
\begin{theorem}\label{thm:record-compression-fork}
Fix \(t_0>0\) before a hypothetical first singular time \(T\), assume the
swirl is nontrivial, and let
\[
 R_j\downarrow0,\qquad t_j\uparrow T,\qquad
 A_j:=A_{R_j}(t_j)\to\infty.
\]
For all sufficiently large \(j\), at least one of the following holds:
\begin{equation}\label{eq:meridional-record}
 \boxed{
 \frac1{R_j}
 \int_{B_{R_j}(x_*)}|u^{\rm mer}(x,t_j)|^2\,\dd x
 >
 \frac{A_j^2}{4}.
 }
\end{equation}
or
\begin{equation}\label{eq:signed-compression-debt}
 \boxed{
 \mathcal A_2(t_0,t_j)
 \ge
 \frac{A_j^2}{24\pi\Gamma_*^2}
 +\frac23\log A_j
 -\frac13\log Z_2(t_0)
 -\log2.
 }
\end{equation}
In particular, failure of the meridional alternative along an infinite
kinetic-record subsequence forces
\(\mathcal A_2(t_0,t_j)\to\infty\).
\end{theorem}

\begin{proof}
Assume \eqref{eq:meridional-record} fails.  Then \cref{lem:swirl-cone} gives
\[
 E_{R_j,\varepsilon_j}^{\theta,\mathrm{in}}
 \ge\frac{A_j^2}{4},
 \qquad
 \varepsilon_j
 =
 2e^{-A_j^2/(8\pi\Gamma_*^2)}.
\]
On the inner cone \(r<\varepsilon_jR_j\),
\(r^{-1}\ge(\varepsilon_jR_j)^{-1}\), so
\[
 \begin{aligned}
 Z_2(t_j)
 &=\int r^{-1}|u^\theta|^2\,\dd x\\
 &\ge
 \frac1{\varepsilon_jR_j}
 \int_{B_{R_j}\cap\{r<\varepsilon_jR_j\}}
 |u^\theta|^2\,\dd x\\
 &=
 \frac1{\varepsilon_j}
 E_{R_j,\varepsilon_j}^{\theta,\mathrm{in}}
 \ge
 \frac{A_j^2}{8}
 e^{A_j^2/(8\pi\Gamma_*^2)}.
 \end{aligned}
\]
Substitute this bound into \eqref{eq:q2-action}.  Expanding the logarithm gives
exactly \eqref{eq:signed-compression-debt}.
\end{proof}

\begin{remark}
The theorem is a branch-routing statement, not a regularity theorem.  It does
not show that the signed critical action is bounded, nor does it exclude the
meridional Type-II record.  Likewise, the parabolic average and density
estimates above do not convert a selected record into the space--time
essential supremum required by a separate regularity criterion.
\end{remark}

\section{Quadratic circulation balance and the scalar-weight mismatch}
\label{sec:circulation-owner}

The circulation equation is homogeneous:
\[
 \partial_t\Gamma+b\cdot\nabla\Gamma
 =\nu\left(\Gamma_{rr}-\frac1r\Gamma_r+\Gamma_{zz}\right).
\]
It therefore carries a quadratic circulation quantity that is not generated by the $G$-source.

\medskip
\noindent\textbf{Quadratic circulation balance.}
\begin{theorem}
For every smooth finite-energy pole-compatible solution for which the boundary
terms vanish,
\[
 \boxed{
 \frac d{dt}\int_{\R_+\times\R}\Gamma^2r\,dr\,dz
 =-2\nu\int_{\R_+\times\R}r|\nabla\Gamma|^2\,dr\,dz.
 }
\]
In particular the quadratic circulation stock is nonincreasing.
\end{theorem}

\begin{proof}
Multiply
\[
 \Gamma_t+b\cdot\nabla\Gamma
 =\nu\left(\Gamma_{rr}-\frac1r\Gamma_r+\Gamma_{zz}\right)
\]
by $2\Gamma r$ and integrate over $\R_+\times\R$.
The time derivative is
\[
 2\int \Gamma\Gamma_t r\,dr\,dz
 =\frac{d}{dt}\int \Gamma^2r\,dr\,dz.
\]
For the transport term,
\[
 2\Gamma\,b\cdot\nabla\Gamma=b\cdot\nabla(\Gamma^2),
\]
so
\[
 \int b\cdot\nabla(\Gamma^2)r\,dr\,dz
 =-\int \Gamma^2
 \bigl[\partial_r(ru^r)+\partial_z(ru^z)\bigr]dr\,dz=0
\]
by axisymmetric incompressibility and the boundary assumptions.

For the radial diffusion we compute
\[
 \begin{aligned}
 \int_0^\infty 2\Gamma r
 \left(\Gamma_{rr}-\frac1r\Gamma_r\right)dr
 &=2\int_0^\infty r\Gamma\Gamma_{rr}\,dr
   -2\int_0^\infty\Gamma\Gamma_r\,dr\\
 &=-2\int_0^\infty r\Gamma_r^2\,dr
   -4\int_0^\infty\Gamma\Gamma_r\,dr.
 \end{aligned}
\]
The last integral is a boundary term:
\[
 -4\int_0^\infty\Gamma\Gamma_r\,dr
 =-2[\Gamma^2]_{0}^{\infty}.
\]
Pole compatibility gives $\Gamma(0,z,t)=0$, and decay gives $\Gamma(\infty,z,t)=0$, so this term vanishes.  The axial diffusion similarly gives
\[
 \int_{\R}2\Gamma r\Gamma_{zz}\,dz
 =-2\int_{\R}r\Gamma_z^2\,dz.
\]
After integration in both variables,
\[
 \frac{d}{dt}\int\Gamma^2r\,dr\,dz
 =-2\nu\int r\bigl(\Gamma_r^2+\Gamma_z^2\bigr)\,dr\,dz,
\]
which is the claimed identity.
\end{proof}

The preceding balance is global only when its quadratic stock is finite.  The next
elementary example is included to prevent a misleading use of the identity as a
finite global event budget under the basic finite-energy assumptions.

\medskip
\noindent\textbf{Finite energy and bounded circulation do not imply finite quadratic circulation stock.}
\begin{proposition}\label{prop:infinite-circulation-stock}
There exist smooth axisymmetric finite-energy velocity fields with bounded circulation
\(\Gamma=ru^\theta\) for which
\[
 \int_{\mathbb R_+\times\mathbb R}\Gamma^2 r\,dr\,dz=\infty.
\]
In particular, finiteness of kinetic energy together with \(\Gamma\in L^\infty\)
does not justify summing the global quadratic-circulation balance over an
arbitrary infinite sequence of localized events.
\end{proposition}

\begin{proof}
Choose a nonzero function \(\chi\in C_c^\infty(\mathbb R)\) and define a pure-swirl
field by
\[
 u^\theta(r,z)=\chi(z)\,r(1+r^2)^{-5/4},
 \qquad
 u^r=u^z=0.
\]
The field is smooth at the axis because \(u^\theta(r,z)=O(r)\) as \(r\downarrow0\),
and it is divergence free because a \(\theta\)-independent pure-swirl field has
zero divergence.

Its three-dimensional kinetic energy is, up to the harmless factor \(2\pi\),
\[
 \int_{\mathbb R}\!\int_0^\infty |u^\theta(r,z)|^2r\,dr\,dz
 =
 \|\chi\|_{L^2(\mathbb R)}^2
 \int_0^\infty r^3(1+r^2)^{-5/2}\,dr.
\]
Near \(r=0\) the radial integrand is \(O(r^3)\), while as \(r\to\infty\) it is
\(O(r^{-2})\).  Hence the integral is finite.

The circulation is
\[
 \Gamma(r,z)
 =
 r u^\theta(r,z)
 =
 \chi(z)\,r^2(1+r^2)^{-5/4}.
\]
The radial factor is continuous, vanishes at \(r=0\), and behaves like
\(r^{-1/2}\) at infinity.  Thus \(\Gamma\in L^\infty\).

On the other hand,
\[
 \int_{\mathbb R}\!\int_0^\infty \Gamma(r,z)^2r\,dr\,dz
 =
 \|\chi\|_{L^2(\mathbb R)}^2
 \int_0^\infty r^5(1+r^2)^{-5/2}\,dr.
\]
The last integrand tends to \(1\) as \(r\to\infty\), so the radial integral
diverges.  This proves the first assertion.

For a bounded meridional region the same quadratic stock is finite, but
localizing the circulation equation introduces flux through the artificial
boundary or collar.  Therefore the global monotonicity identity cannot be
used as a finite event-counting budget unless the global stock is separately
known to be finite.
\end{proof}

On a sign-definite circulation region write $\phi=\log|\Gamma|$ and, for
$p>0$ and radial exponent $k$, set
\[
 \rho_{p,k}=|\Gamma|^p r^{-k},\qquad
 c_p=b-p\nu\nabla\phi.
\]

\medskip
\noindent\textbf{Weighted circulation-current identity.}
\begin{proposition}\label{prop:weighted-current}
One has
\[
\boxed{
\partial_t\rho_{p,k}+\nabla\!\cdot(\rho_{p,k}c_p)
=\rho_{p,k}\left[(k+1)q
+p\nu\left((k-1)\frac{\phi_r}{r}-(p-1)|\nabla\phi|^2\right)\right],
}
\]
where $q=-u^r/r$.  The current action
\[
 \iint\rho_{p,k}|c_p|^2\,dr\,dz\,dt
\]
has physical scale weight $R^{2-k}$, independently of $p$.  Compression disappears from the continuity equation only when $k=-1$, whereas the action is scale invariant only when $k=2$.
\end{proposition}
\begin{proof}
On a sign-definite component the logarithm is smooth and
\[
 \nabla\Gamma=\Gamma\nabla\phi.
\]
Dividing the circulation equation by $\Gamma$ and using
\[
 \frac{\Gamma_{rr}}{\Gamma}=\phi_{rr}+\phi_r^2,
 \qquad
 \frac{\Gamma_{zz}}{\Gamma}=\phi_{zz}+\phi_z^2
\]
gives
\[
 D_t\phi
 =\nu\left(\Delta\phi-\frac{\phi_r}{r}+|\nabla\phi|^2\right),
\]
where $D_t=\partial_t+b\cdot\nabla$ and $\Delta=\partial_{rr}+\partial_{zz}$ in the meridional variables.

Since
\[
 \log\rho_{p,k}=p\phi-k\log r,
\]
we have
\[
 D_t\log\rho_{p,k}
 =pD_t\phi-k\frac{u^r}{r}
 =pD_t\phi+kq.
\]
Furthermore, meridional incompressibility in cylindrical coordinates gives
\[
 \nabla\cdot b=\partial_ru^r+\partial_zu^z=-\frac{u^r}{r}=q.
\]
Hence
\[
 \partial_t\rho_{p,k}+\nabla\cdot(\rho_{p,k}b)
 =\rho_{p,k}\bigl(pD_t\phi+(k+1)q\bigr).
\]
Now replace $b$ by $c_p=b-p\nu\nabla\phi$.  Since
\[
 \nabla\cdot(\rho_{p,k}c_p)
 =\nabla\cdot(\rho_{p,k}b)
 -p\nu\nabla\cdot(\rho_{p,k}\nabla\phi),
\]
we obtain
\[
 \begin{aligned}
 \partial_t\rho_{p,k}+\nabla\cdot(\rho_{p,k}c_p)
 ={}&\rho_{p,k}\bigl(pD_t\phi+(k+1)q\bigr)\\
 &-p\nu\rho_{p,k}
 \left(\Delta\phi+\nabla\log\rho_{p,k}\cdot\nabla\phi\right).
 \end{aligned}
\]
But
\[
 \nabla\log\rho_{p,k}
 =p\nabla\phi-\frac{k}{r}e_r.
\]
Substituting the formula for $D_t\phi$ and cancelling $p\nu\Delta\phi$ gives
\[
 \partial_t\rho_{p,k}+\nabla\cdot(\rho_{p,k}c_p)
 =\rho_{p,k}\left[(k+1)q
 +p\nu\left((k-1)\frac{\phi_r}{r}-(p-1)|\nabla\phi|^2\right)\right].
\]

For the scaling statement, use Navier--Stokes scaling at physical length $R$.  The circulation $\Gamma=ru^\theta$ is scale invariant, so $|\Gamma|^p$ contributes no power of $R$.  The factor $r^{-k}$ contributes $R^{-k}$; the meridional space-time element $dr\,dz\,dt$ contributes $R^4$; and the current velocity $c_p$, like $b$, has size $R^{-1}$, so $|c_p|^2$ contributes $R^{-2}$.  Therefore
\[
 \iint\rho_{p,k}|c_p|^2
\]
scales like
\[
 R^{-k}R^4R^{-2}=R^{2-k}.
\]
This is independent of $p$.  In the reaction term of the continuity identity, the compression coefficient is $k+1$, so it vanishes exactly for $k=-1$.  The action is scale zero exactly for $2-k=0$, namely $k=2$.  The two requirements are therefore incompatible.
\end{proof}

\medskip
\noindent\textbf{No scalar circulation weight gives both properties.}
\begin{corollary}
No circulation power $|\Gamma|^p$ and radial power $r^{-k}$ simultaneously
produces both a compression-free global continuity law and a scale-zero
current action.  Thus attempts to close recurrence by changing only the scalar
circulation weight are structurally insufficient.
\end{corollary}
\begin{proof}
By \cref{prop:weighted-current}, compression-free transport requires $k=-1$, while scale invariance of the current action requires $k=2$.  Since no real number $k$ can satisfy both conditions, no choice of the circulation power $p$ repairs the mismatch.  The independence of the action exponent from $p$ shows that changing only the power of $|\Gamma|$ cannot alter this conclusion.
\end{proof}

\begin{remark}
The intrinsic source-capacity estimate
\[
 \mathcal S_R\le C\Gamma_*^2\Theta_R
\]
shows that a strict source-active tile with $\mathcal S_R\ge s_*>0$ has a
positive quadratic-circulation occupancy floor.  Such a packet cannot evade
the circulation balance merely by making the $\Gamma^2$ occupancy vanish.  This does not
eliminate genuinely low-occupancy diffuse tails, changing ownership, or
boundary/collar reassembly.
\end{remark}

\begin{remark}
For a proportional intrinsic block, the physical circulation transport/pair-work
charge has positive scale homogeneity (schematically $O(R)$).  A supergeometric
sequence of descendant scales can therefore have finite total physical cost while
retaining order-one normalized events.  The circulation balance is a routing and
nonlayering theorem, not by itself a Zeno-proof of regularity.
\end{remark}

\section{Compactness of rescaled strain generators}
Let $I_n=[t_n,t_n+\theta_n]$ and
\[
\widehat C_n(s)=\theta_n C_n(t_n+\theta_ns),\qquad 0<s<1.
\]
Then exactly
\[
\|\widehat C_n\|_{L^1}=\int_{I_n}\|C_n\|,
\qquad
\operatorname{Var}(\widehat C_n)=\theta_n\operatorname{Var}_{I_n}(C_n).
\]

\medskip
\noindent\textbf{Strong active-clock compactness.}
\begin{lemma}\label{lem:bvcompact}
If
\[
\int_{I_n}\|C_n\|\le M,
\qquad
\theta_n\operatorname{Var}_{I_n}(C_n)\le M_{BV},
\]
then, after a subsequence,
\[
\widehat C_n\to C_*
\quad\text{strongly in }L^1(0,1).
\]
Moreover if $Z_C'=CZ_C$, $Z_C(0)=I$, then
\[
\boxed{
\|Z_C(1)-Z_D(1)\|
\le e^{\|C\|_1+\|D\|_1}\|C-D\|_1.
}
\]
Hence ordered endpoint maps converge along the subsequence.
\end{lemma}
\begin{proof}
By the change of variables $t=t_n+\theta_ns$,
\[
 \|\widehat C_n\|_{L^1(0,1)}
 =\int_0^1\theta_n\|C_n(t_n+\theta_ns)\|\,ds
 =\int_{I_n}\|C_n(t)\|\,dt
 \le M.
\]
Similarly, total variation transforms according to
\[
 \operatorname{Var}_{(0,1)}(\widehat C_n)
 =\theta_n\operatorname{Var}_{I_n}(C_n)
 \le M_{BV}.
\]
Because the matrix space $M_2(\R)$ is finite dimensional, componentwise BV compactness applies.  A uniformly bounded set in $BV(0,1;M_2)$ is relatively compact in $L^1(0,1;M_2)$.  Hence there is a subsequence and a limit $C_*\in BV(0,1;M_2)$ such that
\[
 \widehat C_n\to C_*
 \qquad\text{strongly in }L^1(0,1).
\]

We now prove the endpoint stability estimate.  Let $Z_C$ and $Z_D$ solve
\[
 Z_C'=CZ_C,
 \qquad
 Z_D'=DZ_D,
 \qquad
 Z_C(0)=Z_D(0)=I.
\]
The integral equations give
\[
 Z_C(t)=I+\int_0^tC(s)Z_C(s)\,ds,
\]
so Gronwall's inequality yields
\[
 \|Z_C(t)\|\le e^{\int_0^t\|C(s)\|ds}\le e^{\|C\|_1},
\]
and similarly $\|Z_D(t)\|\le e^{\|D\|_1}$.
Subtracting the two integral equations,
\[
 Z_C(t)-Z_D(t)
 =\int_0^tC(s)(Z_C-Z_D)(s)\,ds
 +\int_0^t(C-D)(s)Z_D(s)\,ds.
\]
Taking norms gives
\[
 \|Z_C(t)-Z_D(t)\|
 \le\int_0^t\|C(s)\|\|Z_C-Z_D\|(s)ds
 +e^{\|D\|_1}\|C-D\|_{L^1(0,t)}.
\]
A second application of Gronwall gives
\[
 \|Z_C(t)-Z_D(t)\|
 \le e^{\|C\|_1+\|D\|_1}\|C-D\|_1.
\]
Evaluating at $t=1$ proves the displayed estimate.  Since the right-hand side tends to zero whenever $C_n\to C_*$ in $L^1$ with uniformly bounded $L^1$ norms, the ordered endpoint maps converge along the selected subsequence.
\end{proof}

\medskip
\noindent\textbf{Compactness of rescaled generators.}
\begin{theorem}\label{thm:bvrigidity}
Let $I_n=[t_n,t_n+\theta_n]$ be a selected sequence.  Assume the generator action and rescaled BV norm are uniformly bounded, and assume that the ordered endpoint response is uniformly nontrivial in the sense that
\[
 \|Z_{\widehat C_n}(1)-I\|\ge\eta_*>0.
\]
Suppose moreover that the action-weighted provenance, age, flow-map, frequency, and contact marks are tight.  Then, after passage to a subsequence, the rescaled generators converge strongly in $L^1(0,1)$, the ordered endpoint maps converge, and the corresponding marked action measures converge weakly to a nonzero Radon measure supported on the tight marked state class.  Consequently loss of this compactness can occur only through unbounded rescaled variation or failure of tightness of at least one of the displayed marks.
\end{theorem}
\begin{proof}
The first conclusion is exactly the compactness statement of \cref{lem:bvcompact}.  After passing to a subsequence,
\[
 \widehat C_n\to C_*
 \qquad\text{strongly in }L^1(0,1).
\]
The endpoint stability estimate then gives
\[
 Z_{\widehat C_n}(1)\to Z_{C_*}(1).
\]
Because
\[
 \|Z_{\widehat C_n}(1)-I\|\ge\eta_*,
\]
continuity implies
\[
 \|Z_{C_*}(1)-I\|\ge\eta_*,
\]
so the limiting ordered response is still nontrivial.

We next record why the limiting marked action measure cannot vanish.  Let
\[
 V_n:=\int_0^1\|\widehat C_n(s)\|\,ds.
\]
For any integrable generator $C$,
\[
 \|Z_C(1)-I\|
 \le e^{\|C\|_1}-1.
\]
Indeed, this follows directly from the integral equation and the bound $\|Z_C(t)\|\le e^{\|C\|_1}$.  Therefore
\[
 \eta_*
 \le e^{V_n}-1,
\]
which implies the uniform positive lower bound
\[
 V_n\ge \log(1+\eta_*)>0.
\]
The assumed action upper bound gives $V_n\le M$.  Thus the action-weighted measures have total mass bounded above and bounded away from zero.

To include the auxiliary information, let $\mathfrak m_n(s)$ denote the collection of provenance, age, flow-map, frequency, and contact marks, and define the finite marked measure schematically by
\[
 d\mu_n(s,m)
 =\|\widehat C_n(s)\|\,ds\,\delta_{\mathfrak m_n(s)}(dm).
\]
The time variable already lies in the compact interval $[0,1]$.  By hypothesis the mark marginals are tight, so the family $(\mu_n)$ is tight after normalization by its uniformly bounded total mass.  Prokhorov's theorem gives a subsequence converging weakly to a finite Radon measure $\mu_*$.  The lower total-mass bound passes to the limit and gives
\[
 \mu_*([0,1]\times\mathcal M)
 \ge\log(1+\eta_*)>0,
\]
so $\mu_*$ is nonzero.  Its mass is carried by the tight marked state class selected by the hypothesis.

Thus bounded action, bounded rescaled variation, and tightness of every displayed mark force a convergent marked subsequence with nontrivial endpoint response.  Taking the contrapositive, failure of such compactness requires either blow-up of the rescaled variation or failure of tightness in at least one of the provenance, age, flow-map, frequency, or contact coordinates.
\end{proof}

\section{A quantitative BV floor for first-level cancellation}
Let $C\in BV([0,\theta];M_2)$ be trace free, put
\[
 V=\int_0^\theta\|C(t)\|\,dt,
 \qquad
 A=\int_0^\theta C(t)\,dt,
\]
and let $Z'=CZ$, $Z(0)=I$.

\begin{proposition}\label{prop:bv-floor}
If $\det Z(\theta)=1$ and $\kappa_2(Z(\theta))\ge K>1$, then
\[
\boxed{
\theta\operatorname{Var}(C)
\ge
\frac{4((\sqrt K-1)-e^V\|A\|)_+}{e^V-1}.
}
\]
In particular, if the first average level $A$ is negligible while the endpoint condition number remains bounded below by $K>1$, then $\theta\operatorname{Var}(C)$ has a strictly positive lower bound.
\end{proposition}
\begin{proof}
Define the centered primitive
\[
 B(t)=\int_0^t\left(C(s)-\frac A\theta\right)ds.
\]
Then $B(0)=B(\theta)=0$.  The one-dimensional Peano estimate for a BV function with zero endpoint average gives
\[
 \|B\|_{L^\infty(0,\theta)}
 \le\frac\theta4\operatorname{Var}(C).
\]

We write
\[
 C(t)=\frac A\theta+B'(t)
\]
in the distributional sense and use the integral equation
\[
 Z(\theta)-I=\int_0^\theta C(t)Z(t)dt.
\]
The constant-average part is bounded by
\[
 \left\|\int_0^\theta\frac A\theta Z(t)dt\right\|
 \le \|A\|\sup_{0\le t\le\theta}\|Z(t)\|
 \le e^V\|A\|.
\]
For the centered part, integration by parts uses $B(0)=B(\theta)=0$:
\[
 \int_0^\theta B'(t)Z(t)dt
 =-\int_0^\theta B(t)Z'(t)dt
 =-\int_0^\theta B(t)C(t)Z(t)dt.
\]
Therefore
\[
 \begin{aligned}
 \left\|\int_0^\theta B'(t)Z(t)dt\right\|
 &\le \|B\|_\infty
 \int_0^\theta\|C(t)\|\|Z(t)\|dt\\
 &\le \|B\|_\infty
 \int_0^\theta\|C(t)\|e^{\int_0^t\|C(s)\|ds}dt\\
 &=\|B\|_\infty(e^V-1).
 \end{aligned}
\]
Using the Peano bound gives
\[
 \|Z(\theta)-I\|
 \le e^V\|A\|
 +\frac\theta4(e^V-1)\operatorname{Var}(C).
\]

It remains to convert the condition-number hypothesis into a lower bound for $\|Z(\theta)-I\|_2$.  Because $\det Z(\theta)=1$, the two singular values are $\sigma$ and $\sigma^{-1}$ with $\sigma\ge1$.  The $2$-norm condition number is $\kappa_2=\sigma^2$, so $\kappa_2\ge K$ implies
\[
 \sigma\ge\sqrt K.
\]
For a unit right singular vector $v$ corresponding to $\sigma$,
\[
 \|(Z-I)v\|
 \ge \|Zv\|-\|v\|
 =\sigma-1
 \ge\sqrt K-1.
\]
Hence
\[
 \|Z(\theta)-I\|_2\ge\sqrt K-1.
\]
Combining the upper and lower estimates yields
\[
 \sqrt K-1
 \le e^V\|A\|
 +\frac\theta4(e^V-1)\operatorname{Var}(C).
\]
After rearranging and taking the positive part of the numerator, we obtain the displayed bound.
\end{proof}

\section{Sharpness and limiting examples}
\begin{itemize}[leftmargin=*]
\item $q=13/14$ lies above $8/9$, so the divergent-cancellation energy is
$R^{5/14}$ and the $8/9$ argument leaves that branch compatible with finite
energy.
\item The $12/13$ pincer explicitly retains older/incoming same-shell
cancellation as a transfer-fed alternative.
\item The $27/29$ theorem is retained as a sharp quantitative subcase.  The broader q-uniform theorem routes compact finite-depth source birth throughout $6/7<q<1$, so $q>27/29$ is not a separate tight endpoint branch.
\item The sharp native-sheet scaling $R_N=N^{-1/5}$ keeps order-one trace
response with bounded circulation only by producing a finer axis/$F^4$ concentration and
diverging quartic stock; it is preemption/output, not a strict descendant.
\item Bounded action with unbounded scaled variation is a genuine non-BV/rate
survivor, not a contradiction from scaling alone.
\end{itemize}

\section{Limitations and open alternatives}
The estimates above are exact or quantitative under the hypotheses stated in their respective sections.  They rule out several compact finite-depth preparations and identify the loss of compactness required by an infinite fresh transition sequence.  They do not bound all possible high-frequency, boundary-fed, or changing-provenance configurations.  The logarithmic-record estimates route sustained swirl energy into meridional records or signed critical compression, but they do not exclude either receiving branch and do not turn isolated record times into a parabolic essential-supremum bound.  In particular, bounded action with unbounded rescaled variation is a genuine noncompact alternative, and the quadratic circulation balance cannot be used as a Zeno-proof by summing a positive physical event cost.

The natural unresolved regimes are therefore those in which normalized frequency, material provenance, boundary/collar entrance, or generator variation fails to remain compact.  No unconditional regularity statement is inferred from this classification.

\appendix
\section{Critical scaling table}
The critical exponents used in the main text follow directly from the homogeneity calculations displayed there.

\end{document}